\documentclass{article}
\usepackage{amsmath}
\usepackage{amsthm}
\usepackage{graphics}
\usepackage{epsfig}
\newtheorem{theorem}{Theorem}[section]

\newtheorem{lem}[theorem]{Lemma}
\newtheorem{prop}[theorem]{Proposition}
\newtheorem{facc}[theorem]{Fact}
\newtheorem{quest}[theorem]{Question}

\theoremstyle{definition}

\newtheorem{rem}[theorem]{Remark}

\numberwithin{equation}{section}
\usepackage{amssymb}
\usepackage{amsfonts}
\usepackage{stmaryrd}
\usepackage{psfrag}
\usepackage{color}
\usepackage{url}
\usepackage{mathtools}
\usepackage{textcomp}
\newcommand{\N}{\mathbb{N}}

\newcommand{\Z}{\mathbb{Z}}
\newcommand{\Q}{\mathbb{Q}}
\newcommand{\C}{\mathbb{C}}

\begin{document}

\title{Linear independence of trigonometric numbers}

\author{Arno Berger\\[2mm] Mathematical and
    Statistical Sciences\\University of Alberta\\Edmonton, Alberta,
  {\sc Canada}}

\maketitle

\begin{abstract}
\noindent
Given any two rational numbers $r_1$ and $r_2$, a necessary and
sufficient condition is established for the three numbers $1$, $\cos
(\pi r_1)$, and $\cos (\pi r_2)$ to be rationally independent. Extending a classical
fact sometimes attributed to I.\ Niven, the result even yields linear independence over larger number
fields. The tools employed in the proof are applicable also in the
case of more than two trigonometric numbers. As an 
application, a complete classification is given of all planar
triangles with rational angles and side lengths each containing at
most one square root. Such a classification was hitherto known only
in the special case of right triangles.
\end{abstract}
\hspace*{6.6mm}{\small {\bf Keywords.} Niven's Theorem, rational (in)dependence, cyclotomic
  polynomial,}\\[-1mm]
\hspace*{25.5mm}{\small real quadratic number field, rational triangle, high school triangle.}

\noindent
\hspace*{6.6mm}{\small {\bf MSC2010.} 11R09, 11R11, 12Y05, 97G60.}

\medskip

\section{Introduction}

Denote the fields of all rational and all complex numbers by $\Q$ and $\C$,
respectively. For every $r\in \Q$ let $N(r)$
be the smallest positive integer for which $rN(r)$ is an
integer; equivalently, if $r=p/q$ with coprime integers $p$ and $q>0$ then
simply $N(r)=q$. Recall that $\C$ is a linear space over any subfield
$K\subset \C$, and hence the notion of a family of complex numbers being
{\em linearly\/} ({\em in}){\em dependent over\/} $K$, or $K$-({\em
  in}){\em dependent\/} for short, is well-defined.
Given any $r\in \Q$, the algebraic properties of
trigonometric numbers such as $\cos(\pi r)$, $\sin (\pi r)$, and $\tan
(\pi r)$ have long been of interest; e.g., see \cite{Carlitz62,
  Lehmer33, Niven56, 
  Olmsted45, Richmond36, Schaumberger74, Underwood21, Underwood22} for time-honoured, and
\cite{Calcut09, Calcut10, Jahnel10, Varona06} for more recent accounts. A classical
fact in this regard, already recorded in \cite{Underwood21} but
sometimes attributed to
\cite{Niven56} as {\em Niven's Theorem\/}, asserts that if the number
$2\cos (\pi r)$ is rational then it is in fact an integer. In other words,
\begin{equation}\label{eq001}
\{ 2 \cos (\pi r) : r \in \Q\} \cap \Q = \{-2,-1,0,1,2\} \, .
\end{equation}
Analogous results exist for $\sin
(\pi r)$ and $\tan (\pi r)$.
The present article is motivated by
(\ref{eq001}) primarily through an equivalent form, stated here as

\begin{facc}\label{fac1}
Let $r\in \Q$. Then the following are equivalent:
\begin{enumerate}
\item The numbers $1$ and $\cos(\pi r)$ are $\Q$-independent;
\item $N(r)\ge 4$.
\end{enumerate}
\end{facc}
 
Fact \ref{fac1} naturally leads to
the question whether it extends in any recognizable form to more than
one trigonometric number. This question does not seem to
have been adressed in the literature. The present article
provides a complete answer in the case of {\em two\/} trigonometric
numbers. However, the tools assembled in the process are applicable
also in the case of more than two numbers, to be studied elsewhere.

To see what an analogue of Fact
\ref{fac1} for two numbers $\cos(\pi r_1)$ and $\cos (\pi r_2)$ with $r_1,r_2\in
\Q$ might look like, note first that clearly those two numbers are
$\Q$-dependent whenever $r_1 - r_2$ or $r_1+r_2$ is an
integer. Moreover,
\begin{equation}\label{eq01}
2\cos (\pi /5 ) -  2\cos (  2 \pi / 5) = 1 \, ,
\end{equation}
so $1, \cos(\pi r_1)$, and $\cos(\pi r_2)$ may be
$\Q$-dependent if $N(r_1)= N(r_2)=5$. As the following theorem,
the main result of this article, shows, there are
no other obstacles to $\Q$-independence. Notice that the theorem is a true
analogue of Fact \ref{fac1} and immediately implies the latter.

\begin{theorem}\label{thm2}
Let $r_1,r_2\in \Q$ be such that neither $r_1 - r_2$ nor $r_1 + r_2$
is an integer. Then the following are equivalent:
\begin{enumerate}
\item The numbers $1$, $\cos(\pi r_1)$, and $\cos (\pi r_2)$ are
  $\Q$-independent;
\item $N(r_j)\ge 4$ for $j\in \{ 1,2\}$, and $\bigl( N(r_1), N(r_2)\bigr)\ne (5,5)$.
\end{enumerate}
\end{theorem}

\noindent
A proof of Theorem \ref{thm2} is given in Section \ref{sec3}. There it
will also be seen that the three numbers in (i) often are linearly independent even
over larger number fields. The proof relies on several elementary facts regarding cyclotomic polynomials that may be of
independent interest. For the reader's convenience, Section
\ref{sec2} recalls these facts, or establishes them in cases where no
reference is known to the author.

As an amusing application of Theorem \ref{thm2} and the tools
assembled for its proof, the final Section \ref{sec4} of this article
provides a complete classification of planar triangles of a certain
type. Concretely, let $0< \delta_1 \le \delta_2 \le \delta_3 < \pi$ be
the angles of (the similarity type of) a non-degenerate planar
triangle $\Delta$, in symbols $\Delta = (\delta_1, \delta_2,
\delta_3)$ with $\delta_1 + \delta_2 +\delta_3 = \pi$. Call $\Delta$
{\em rational\/} if $\delta_j/\pi \in \Q$ for all $j$. (Several other,
non-equivalent variants of the term ``rational triangle'' can be found in the
literature \cite{wiki1}; they will not be employed here.) Given any
rational triangle $\Delta$, denote by $N(\Delta)$ the
least common multiple of the numbers $N(\delta_j/\pi)$; with this,
$\Delta = \pi  (n_1,n_2,n_3)/N(\Delta)$, where the positive integers
$n_j= \delta_j N(\Delta)/\pi$ have no common factor. Considering a concrete
realization of $\Delta$, for each $j$ let $\ell_j$ be the length of
the side vis-\`a-vis the angle $\delta_j$. While these lengths
are determined by
$\Delta=(\delta_1,\delta_2,\delta_3)$ only up to scaling, i.e., up to
simultaneous multiplication by a positive factor, all ratios
$\ell_j/\ell_k$ are uniquely determined since, by the law of sines,
$$
\ell_1 : \ell_2 : \ell_3 = \sin \delta_1 : \sin \delta_2 : \sin
\delta_3 \, ;
$$
here and throughout, an equality $x_1:x_2:x_3 = y_1:y_2:y_3$ with
real numbers $x_j,y_j>0$ is understood to mean that both
$x_1/x_2 = y_1/y_2$ and $x_2/x_3 = y_2/y_3$. Notice that usage of
an expression $x_1:x_2:x_3$ does not automatically imply that $x_1/x_2
= x_2/x_3$, i.e., the numbers $x_1$, $x_2$, and $x_3$ need not be in
``continued proportion''.

Arguably the simplest triangles are those for which $\ell_1 : \ell_2 :\ell_3 = r_1:r_2:r_3$ with
rational numbers $0<r_1 \le r_2 \le r_2$; clearly, all $r_j$ can be
assumed to be integers in this case. For example,
$\ell_1:\ell_2:\ell_3 = 3:4:5$ is the first
Pythagorean right triangle. It is well known, however, that for {\em
  rational\/} triangles this simple situation occurs only in the
equilateral case. More formally, if $\Delta$ is rational with $\ell_1
:\ell_2 :\ell_3 = r_1:r_2:r_3$ then $r_1 = r_2 = r_3$ and $\Delta =
\pi (1,1,1)/3$; e.g., see \cite[Cor.6]{Calcut09} or
\cite[p.228]{ConwayGuy}. To identify a slightly wider class of triangles that
may rightfully be considered ``simple'', in the
spirit of the charming article \cite{Calcut10} call 
$\Delta$ a {\em high school triangle\/} if $\ell_1 : \ell_2 :\ell_3 =
x_1:x_2:x_3$ where each $x_j$ is either a rational number or a (real)
quadratic irrational, i.e., $x_j = r_j + s_j\sqrt{d_j}$ with
$r_j,s_j\in \Q$ and integers $d_j\ge 2$. Again, all $r_j$ and $s_j$ can
be assumed to be integers. Thus, informally put, a high school
triangle can be realized in such a way that each side-length is expressed using only integers and at most one square-root
symbol. For instance, $\Delta = \pi
(1,1,2)/4$ is both rational and a high school triangle since $\ell_1 :
\ell_2 :\ell_3 = 1:1:\sqrt{2}$. Observe that this $\Delta$ is also
right and hence an example of a right, rational high school
triangle. As demonstrated in \cite{Calcut10}, there exist altogether only two other triangles
of this nature, namely
$$
\Delta =  \pi (1,2,3) /6 \quad \mbox{\rm and} \quad
\Delta =  \pi (1,5,6) / 12 \, .
$$
Thus, within the countably infinite family of all right, rational
triangles, only three are high school triangles. This remarkable
scarcity prevails even without the assumption of a right angle: Very few
rational triangles are high school triangles.
Utilizing Theorem \ref{thm2} and the tools employed in its proof, it
is straightforward to demonstrate this rigorously, by establishing the following {\em
  complete classification of rational high school triangles\/}
which will be discussed in detail in Section \ref{sec4}.

\begin{theorem}\label{thm3}
Let $\Delta$ be a rational triangle. Then the
following are equivalent:
\begin{enumerate}
\item $\Delta$ is a high school triangle;
\item $N(\Delta)\in \{3,4,5,6,12\}$.
\end{enumerate}
\end{theorem}

\noindent
A brief counting exercise reveals that condition (ii) in
Theorem \ref{thm3} identifies exactly 14 different triangles, of which
seven are isosceles (including the equilateral as well as the right triangle
$ \pi (1,1,2)/4$ from above),
three are right (and hence precisely the ones found in \cite{Calcut10}), and five are neither; see Figure \ref{fig1}.

\begin{figure}[ht]
%
\psfrag{t1a}[]{\small $\pi (1,1,1)/3$}
\psfrag{t1b}[]{\small $1:1:1$}
\psfrag{t2a}[]{\small $\pi (1,1,2)/4$}
\psfrag{t2b}[]{\small $1:1:\sqrt{2}$}
\psfrag{t3a}[]{\small $\pi (1,1,3)/5$}
\psfrag{t3b}[]{\small $2:2:\sqrt{5} + 1$}
\psfrag{t4a}[]{\small $\pi (1,2,2)/5$}
\psfrag{t4b}[]{\small $\sqrt{5}-1:2:2$}
\psfrag{t5a}[]{\small $\pi (1,1,4)/6$}
\psfrag{t5b}[]{\small $1:1:\sqrt{3}$}
\psfrag{t6a}[]{\small $\pi (1,1,10)/12$}
\psfrag{t6b}[]{\small $\sqrt{2}:\sqrt{2}:\sqrt{3}+1$}
\psfrag{t7a}[]{\small $\pi (2,5,5)/12$}
\psfrag{t7b}[]{\small $\sqrt{3} -1:\sqrt{2}:\sqrt{2}$}
\psfrag{t8a}[]{\small $\pi (1,2,3)/6$}
\psfrag{t8b}[]{\small $1:\sqrt{3}:2$}
\psfrag{t9a}[]{\small $\pi (1,2,9)/12$}
\psfrag{t9b}[]{\small $\sqrt{3}-1:\sqrt{2}:2$}
\psfrag{t10a}[]{\small $\pi (1,3,8)/12$}
\psfrag{t10b}[]{\small $\sqrt{3} - 1 : 2 :\sqrt{6}$}
\psfrag{t11a}[]{\small $\pi (1,4,7)/12$}
\psfrag{t11b}[]{\small $\sqrt{3} -1 : \sqrt{6}:  \sqrt{3} + 1$}
\psfrag{t12a}[]{\small $\pi (1,5,6)/12$}
\psfrag{t12b}[]{\small $\sqrt{3}-1 :\sqrt{3}+1 :2\sqrt{2}$}
\psfrag{t13a}[]{\small $\pi (2,3,7)/12$}
\psfrag{t13b}[]{\small $\sqrt{2}:2:\sqrt{3}+1$}
\psfrag{t14a}[]{\small $\pi (3,4,5)/12$}
\psfrag{t14b}[]{\small $2 :\sqrt{6}:\sqrt{3} + 1$}
\includegraphics{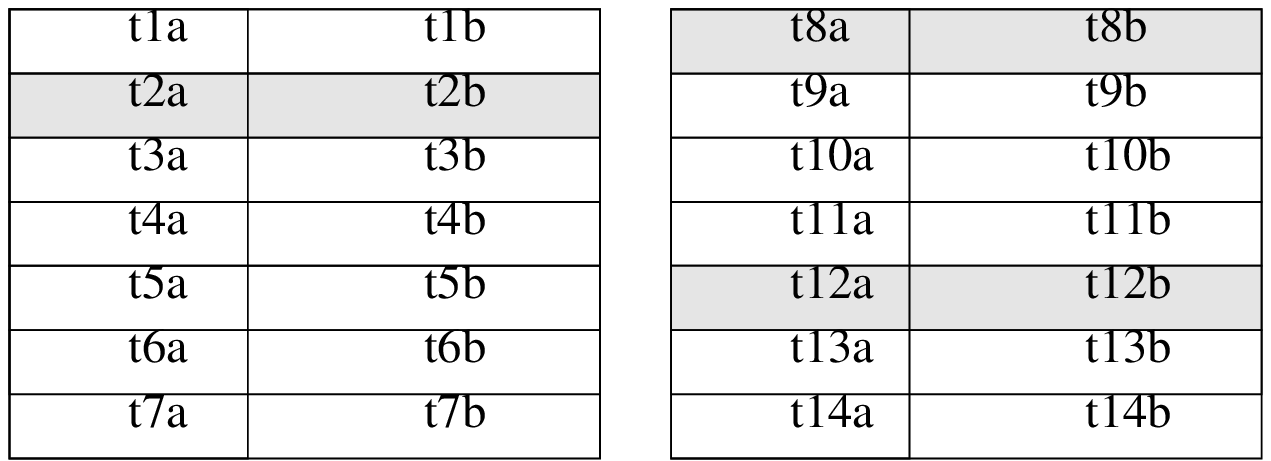}
\caption{As a consequence of Theorem \ref{thm3}, there exist exactly 14
  rational high school triangles, of which seven are isosceles (left
  table), three are right (grey boxes; cf.\ \cite{Calcut10}), and five
  are neither; see Section \ref{sec4} for details.}\label{fig1}
\end{figure}

\begin{rem}
In elementary geometry, angles are often measured in degrees ($^{\circ}$) rather
than radians. Rationality in this context simply means that all angles
are rational numbers, and Theorem \ref{thm3} is readily
reformulated: When measured in degrees, a rational triangle is a
high school triangle if and only if each angle is an integer multiple
of either $15^{\circ}$ or $36^{\circ}$.
\end{rem}

\section{Cyclotomic and other polynomials}\label{sec2}

Denote the sets of all positive integers and all integers by $\N$ and
$\Z$, respectively. For every $n\in \N$ let $\Phi_n = \Phi_n(z)$ be the
$n$-th cyclotomic polynomial,
\begin{equation}\label{eq21}
\Phi_n(z) = \prod\nolimits_{1\le j \le n: {\rm gcd}(j,n) = 1} \left(
z- e^{2\pi \imath j/n}
\right) \, ;
\end{equation}
thus for example
$$
\Phi_1(z) = z - 1 \, , \quad
\Phi_2(z) = z + 1 \, , \quad
\Phi_3(z) = z^2 + z + 1 \,  , \quad
\Phi_4(z) = z^2 + 1 \, .
$$
It is well known that each $\Phi_n$ is monic with integer
coefficients, is irreducible over $\Q$, and has degree $\varphi(n)$, where
$\varphi$ denotes the Euler totient function. For $n\ge 2$
the polynomial $\Phi_n$ is also palindromic, i.e., $\Phi_n(z^{-1}) =
z^{-\varphi (n)} \Phi_n(z)$. The coefficients of $\Phi_n$ are
traditionally labelled $a(j,n)$, thus
\begin{equation}\label{eq21a}
\Phi_n(z) = \sum\nolimits_{j=0}^{\varphi(n)} a(j,n) z^{\varphi(n) - j}
\, ;
\end{equation}
in addition, let $a(j,n)=0$ whenever $j>\varphi(n)$, so that $a(j,n)$ is
defined for all $n\in \N$ and $j\ge 0$. (For later convenience, the
labelling in (\ref{eq21a}) is a reversal of the traditional
one.) The integers $a(j,n)$ are objects of
great combinatorial interest and have been studied extensively; e.g., see
\cite{Bachman93} and the references therein. Only a few specific properties of
cyclotomic polynomials are needed for the purpose of this article and will
now be reviewed; for comprehensive accounts the reader is referred, e.g.,
to \cite[Ch.V.8]{Hung74}, \cite[\S 13]{Kunz}, or \cite[\S 11]{Stroth}.

First observe that while the values of $|a(j,n)|$ may be large for
large $n$ and the appropriate $j$, the four leading coefficients of $\Phi_n$
only attain values in $\{-1,0,1\}$, and in fact exhibit patterns that are even more
restricted. To state this precisely, call an integer $k$ {\em
squarefree\/} if $p^2 \nmid  k$, i.e., $k$ is not divisible by $p^2$,
for any prime number $p$.

\begin{lem}\label{lem22}
Assume $n\in \N$ is squarefree. Then $a(0,n)=1$, and the coefficient triple
$\bigl( a(1,n), a(2,n), a(3,n)\bigr)$ has exactly one of the following
eight values:
$$
(1,1,1), \:
(1,1,0) , \:
(1,0,0), \:
(1,0,-1) , \:
(-1,1,0), \:
(-1,1,-1) , \:
(-1,0,1), \:
(-1,0,0) \, .
$$
\end{lem}

\begin{proof}
By (\ref{eq21}), clearly $a(0,n)=1$ for all $n$. The cases of $n=1$,
$2$, and $3$ yield the
triples $(-1,0,0)$, $(1,0,0)$, and $(1,1,0)$, respectively, all of
which are listed in the statement of the lemma. Hence
assume $n\ge 5$ from now on. Since $n$
is squarefree, there exist $m\in \N$ and prime numbers $p_1> \ldots >
p_m$ such that $n= \prod_{j=1}^m p_j$. 

Assume first that $p_m \ge 5$, and for convenience let $\varphi_j  =
\varphi(p_1 \cdots p_j)$ as well as $a_j= a(1, p_1 \cdots p_j)$,
$b_j= a(2,p_1 \cdots p_j)$, and $c_j = a(3,p_1 \cdots p_j)$ for $j \in
\{ 1, \ldots , m\}$. Thus
$$
\Phi_{p_1 \cdots p_j}(z) = z^{\varphi_j} + a_j z^{\varphi_j - 1} + b_j
z^{\varphi_j - 2} + c_j z^{\varphi_j - 3} + \Psi_j(z) \, ,
$$
with the appropriate polynomial $\Psi_j$ of degree less than $\varphi_j
- 3$. From
$$
\Phi_{p_1}(z) = z^{p_1-1} + z^{p_1 - 2} + \ldots + z + 1 \, ,
$$
it is clear that $\varphi_1=p_1 - 1$ and $a_1 = b_1 = c_1 = 1$. On the other hand,
\begin{align*}
\Phi_{p_1 \cdots p_j p_{j+1}}(z) & = \frac{\Phi_{p_1 \cdots
    p_j}(z^{p_{j+1}})}{\Phi_{p_1 \cdots p_j}(z)} \\
& = \frac{z^{p_{j+1}\varphi_j} + a_j z^{p_{j+1}(\varphi_j - 1) } + b_j z^{p_{j+1}(\varphi_j -
    2)} + c_j z^{p_{j+1}(\varphi_j - 3)} + \Psi_j
  (z^{p_{j+1}})}{z^{\varphi_j} + a_j z^{\varphi_j - 1} + b_j
  z^{\varphi_j - 2} + c_j z^{\varphi_j - 3} + \Psi_j(z)} \, ,
\end{align*}
which, together with long division and the fact that $p_{j+1}\ge 5$, leads to
\begin{align*}
\Phi_{p_1 \cdots p_j p_{j+1}}(z) & = z^{  (p_{j+1} - 1) \varphi_j } - a_j
z^{ (p_{j+1}-1) \varphi_j -1} +
(a_j^2 - b_j) z^{  (p_{j+1} - 1) \varphi_j - 2} +\\
& \quad + 
(2a_jb_j - a_j^3 - c_j) z^{  (p_{j+1}-1) \varphi_j -3} + \Psi_{j+1} (z)
\, ,
\end{align*}
and hence in turn yields the recursion $\varphi_{j+1} = (p_{j+1} -
1)\varphi_j$ and
\begin{equation}\label{eq221}
a_{j+1} = - a_j \, , \quad
b_{j+1} = a_j^2 - b_j \, , \quad
c_{j+1} = 2 a_jb_j  - a_j^3 - c_j \, .
\end{equation}
Using (\ref{eq221}) with $(a_1,b_1,c_1)=(1,1,1)$ shows that the
triple $(a_j,b_j,c_j)$ can have only two different values, namely
$(1,1,1)$ if $j$ is odd, and $(-1,0,0)$ if $j$ is even. 

Next assume that $p_m = 3$ and hence $m\ge 2$. In this case,
(\ref{eq221}) remains valid for $j \in \{ 1, \ldots , m-2\}$, yet for $j=m-1$
it has to be replaced with
\begin{equation}\label{eq222}
a_{m} = - a_{m-1} \, , \quad
b_{m} = a_{m-1}^2 - b_{m-1} \, , \quad
c_{m} = 2 a_{m-1}b_{m-1} + a_{m-1} - a_{m-1}^3 - c_{m-1}\, .
\end{equation}
Recall from above that $(a_{m-1}, b_{m-1}, c_{m-1})$ equals either
$(1,1,1)$ or $(-1,0,0)$. By (\ref{eq222}), therefore, the value of
$(a_m,b_m,c_m)$ is either $(-1,0,1)$ or $(1,1,0)$.

Finally, if $p_m = 2$ then again $m\ge 2$, and the identity 
$$
\Phi_n(z) = \Phi_{p_1 \cdots p_{m-1} 2} (z) = \Phi_{p_1 \cdots p_{m-1}} (-z)
$$
implies that $a_m = - a_{m-1}$, $b_m = b_{m-1}$, and $c_m =
-c_{m-1}$. This yields the remaining four possible values for $(a_m,b_m,c_m)$.
\end{proof}

From (\ref{eq21}), it is easy to see that $\Phi_{mn} (z) = \Phi_n
(z^m)$, provided that every prime number
dividing $m$ also divides $n$. With this, Lemma \ref{lem22} restricts the possible
values for the leading coefficients of $\Phi_n$ even in cases where $n$ is
not squarefree.

\begin{lem}\label{lem23}
Assume $n\in \N$ is not squarefree. Then $a(0,n)=1$, and the coefficient triple
$\bigl( a(1,n), a(2,n), a(3,n)\bigr)$ has exactly one of the following
five values:
$$
(0,1,0), \:
(0,0,1) , \:
(0,0,0), \:
(0,0,-1), \:
(0,-1, 0) \, .
$$
\end{lem}

\begin{proof}
Pick any prime number $p$ with $p^2 \mid n$. The assertion follows
immediately from the fact that
$$
\Phi_n(z) = \Phi_{p \cdot n/p} (z) = \Phi_{n/p}(z^p) \, ,
$$
which, together with Lemma \ref{lem22} and the notation adopted in its
proof, implies that the triple $\bigl( a(1,n), a(2,n), a(3,n)\bigr)$
equals either $(0,a_m,0)$, $(0,0,a_m)$, or $(0,0,0)$; recall that $a_m
\in \{-1,0,1\}$.
\end{proof}

\begin{rem}\label{rem24}
(i) Notice that for every squarefree $n\in \N$ the coefficient
$a(1,n)$ equals $1$ or $-1$, depending on whether $n$ has an odd or an
even number of prime factors; if $n$ is not squarefree then
$a(1,n)=0$. Thus simply $a(1,n)= - \mu (n)$, with $\mu$ denoting the
M\"{o}bius function \cite[\S 16.3]{Hardy}.

(ii) Put together, Lemmas \ref{lem22} and
\ref{lem23} allow for a total of 13 possible patterns
for the four leading coefficients of $\Phi_n$. Each of
those patterns already occurs for $n \le 30$, as well as for
infinitely many $n\in \N$ thereafter.
\end{rem}

Another fact relevant in what follows is that
the actual value of $\Phi_n(\imath)$ can easily be computed.
Recall that $\Z[\imath]= \{k +\imath l : k,l \in \Z \}$ denotes the ring of
Gaussian integers.

\begin{lem}\label{lem25}
Let $n\in \N$. Then $\Phi_n(\imath)\in \Z[\imath]$, and the following holds:
\begin{enumerate}
\item If $4\nmid n$ then $ \Phi_n
  (\imath) \in \{-1, -1+\imath , -\imath, \imath, 1, 1+\imath\}$;
\item If $4\mid n$ then
$$
\Phi_n(\imath) = \left\{
\begin{array}{cll}
0 & & \mbox{\rm if } n = 4 \, , \\
p & & \mbox{\rm if $n = 4 p^j$ for some prime number $p$ and $j\in \N$} \, , \\[0.5mm]
1 & & \mbox{\rm otherwise}\, .
\end{array}
\right.
$$
\end{enumerate}
\end{lem}

\begin{proof}
Since $\Phi_n$ has integer coefficients, clearly $\Phi_n(\imath)\in
\Z[\imath]$ for all $n$. Also, with $\Phi_1(\imath) = -1 + \imath$,
$\Phi_2(\imath) = 1 + \imath$, and $\Phi_3(\imath) = \imath$,
evidently (i) holds for $n\in \{1,2,3\}$. From now on, therefore, let
$n\ge 4$. Recall that 
\begin{equation}\label{eq251}
z^n - 1 = \prod\nolimits_{1\le j \le n : j \mid n} \Phi_j (z) =
\Phi_1(z) \prod\nolimits_{2\le j \le n : j \mid n} \Phi_j (z) \, .
\end{equation}

To establish (i), assume first that $n$ is odd. In this case, (\ref{eq251}) implies that
the integer $|\Phi_1(\imath)|^2|\Phi_n (\imath)|^2= 2|\Phi_n
(\imath)|^2$ divides $|\imath^n - 1|^2 = 2$, hence $|\Phi_n (\imath)|
=1$, and $\Phi_n (i) \in \{-1,-\imath, \imath, 1\}$.
Next assume that $n\in 2+4\Z$. Now (\ref{eq251}) yields
$$
-2 = (-1+\imath) (1+\imath) \prod\nolimits_{3\le j \le n : j\mid n}
\Phi_j(\imath) \, ,
$$
and hence again $|\Phi_n(\imath)| = 1$. This proves (i).

To establish (ii), consider the case of $n\in 4 \Z$. Plainly $\Phi_4(\imath)
=0$, so henceforth assume $n \ge 8$. There exist
$m\in \N$, prime numbers $p_1> \ldots > p_m$, and $k_1, \ldots ,
k_m\in \N$ such that $n = 4 \prod_{j=1}^m p_j^{k_j}$. If $p_m\ge 3$
then
$$
\Phi_n(\imath) = \Phi_{2p_1 \cdots p_m} \left( \imath^{2 p_1^{k_1 -1}
    \cdots p_m^{k_m - 1}}\right) = \Phi_{2p_1 \cdots p_m} (-1) =
\Phi_{p_1 \cdots p_m}(1) \, .
$$
Thus $\Phi_n(\imath) = p_1$ if $m=1$, and otherwise
$$
\Phi_n(\imath) = \Phi_{p_1 \cdots p_m}(1) = \frac{\Phi_{p_1 \cdots
    p_{m-1}}(1^{p_m})}{\Phi_{p_1 \cdots p_{m-1}}(1)} = 1 \, .
$$
Similarly, if $p_m = 2$ then
$$
\Phi_n (\imath) = \Phi_{p_1 \cdots p_m} \left( \imath^{p_1^{k_1 - 1}
    \cdots p_{m-1}^{k_{m-1}-1} p_m^{k_m + 1}}\right) = \Phi_{p_1
  \cdots p_m} (1) \, ,
$$
and again $\Phi_n(\imath)= 2 = p_m$ if $m=1$, and $\Phi_n(\imath) = 1$ otherwise.
\end{proof}

\begin{rem}\label{rem25a}
Lemma \ref{lem25}(i) allows for a total of six possible values for
$\Phi_n(\imath)$. While the two values $-1+\imath$ and $1+\imath$
occur only for $n=1$ and $n=2$, respectively, each of the other four
values already occurs for $n\le 21$, as well as for infinitely many
$n\in \N$ thereafter.
\end{rem}

Finally, it will matter later on whether the polynomial
$\Phi_n$, which is irreducible over $\Q$, remains irreducible when $\Q$
is replaced with a larger field, in particular with a real quadratic number
field. Consider, therefore, any squarefree integer $d\ge 2$, and let
$\widehat{d}$ be the discriminant of the number field $\Q \bigl( \sqrt{d}
\bigr)$, that is,
$$
\widehat{d} = \left\{
\begin{array}{ccl}
d & & \mbox{\rm if } d \in 1 + 4 \Z \, , \\[0.5mm]
4d & & \mbox{\rm if } d \in \{2,3\} + 4\Z \, .
\end{array}
\right.
$$

\begin{lem}\label{lem26}
Let $n\in \N$, and assume the integer $d\ge 2$ is squarefree. Then the
polynomial $\Phi_n$ is irreducible over $\Q \bigl( \sqrt{d} \bigr)$ if and only if
$\widehat{d} \nmid n$.
\end{lem}

\begin{proof}
Since the asserted equivalence clearly holds for $n\in \{1,2\}$, let
$n\ge 3$ throughout. For convenience, for every $m\in \N$ denote by
$K_m$ the smallest subfield of $\C$ containing $e^{2\pi \imath /m}$.

Observe first that the irreducibility of $\Phi_n$ over $\Q\bigl( \sqrt{d}\bigr)$ is equivalent to $\Q\bigl( \sqrt{d} \bigr)\cap
K_n = \Q$. Indeed, if $\Q\bigl( \sqrt{d} \bigr) \subset K_n $ then $\bigl[\Q
\bigl( \sqrt{d} , e^{2\pi \imath/n}\bigr): \Q \bigl( \sqrt{d} \bigr)\bigr] = \frac12 [K_n :\Q]=
\frac12 \varphi(n)< \varphi(n)$, showing that $\Phi_n$ cannot be
irreducible over $\Q\bigl( \sqrt{d} \bigr)$. If, on the other hand,
$\Q \bigl( \sqrt{d}\bigr)\not \subset K_n$ then
$\Q\bigl( \sqrt{d} \bigr)\cap K_n = \Q$ and $\bigl[K_n
\bigl(\sqrt{d}\bigr): K_n \bigr]=2$. With this,
$$
2 \bigl[
\Q \bigl(\sqrt{d}, e^{2\pi\imath / n} \bigr) : \Q \bigl( \sqrt{d} \bigr)
\bigr] =
\bigl[
K_n \bigl( \sqrt{d}\bigr) :\Q
\bigr] = 
\bigl[
K_n (\sqrt{d}) : K_n
\bigr] \cdot [K_n : \Q] = 2 \varphi(n) \, ,
$$
hence $\bigl[\Q
\bigl( \sqrt{d} , e^{2\pi \imath/n}\bigr): \Q \bigl( \sqrt{d}
\bigr)\bigr]=\varphi(n)$, which shows that $\Phi_n$ is
irreducible over $\Q\bigl( \sqrt{d} \bigr)$.

It remains to verify that the properties $\Q\bigl( \sqrt{d} \bigr)\cap
K_n = \Q$ and $\widehat{d}\nmid n$ indeed are equivalent. To this end, recall that
$\Q\bigl( \sqrt{d} \bigr)\subset K_{\widehat{d}}$. In fact, $m=\widehat{d}$ is the
smallest $m\in \N$ such that $\Q\bigl( \sqrt{d} \bigr)\subset
K_m$; e.g., see \cite[Cor.VI.1.2]{janusz73}. Thus, if
$\widehat{d}\mid n$ then $\Q\bigl( \sqrt{d} \bigr)\subset K_{\widehat{d}}\subset
K_n$. Conversely, assume $\widehat{d}\nmid n$ and suppose that $\Q\bigl( \sqrt{d} \bigr)\subset
K_n$. Then $\Q\bigl( \sqrt{d} \bigr)\subset K_{\widehat{d}}\cap K_n = K_m$ with
$m = \mbox{\rm gcd}\bigl(\widehat{d},n \bigr)< \widehat{d}$; e.g., see
\cite[(11.24)]{Stroth}. This contradiction proves that $\Q\bigl( \sqrt{d} \bigr)\cap K_n = \Q$ whenever
$\widehat{d}\nmid n$.
\end{proof}

Using the properties of cyclotomic polynomials stated above, it is now
straightforward to identify the minimal polynomials over $\Q$ of
trigonometric algebraic numbers such as $\cos(\pi r)$ or $\sin (\pi r)$, with $r\in
\Q$. Implicitly, this was done already in \cite{Lehmer33}. However, for
the proof of Theorem \ref{thm2} given in the next section, more explicit information about these
polynomials is required. This information is gathered here, starting
from the following well-known fact \cite[Exc.13.17]{Kunz}.

\begin{prop}\label{prop26a}
For every integer $n\ge 0$ there exists a unique polynomial $R_n =
R_n(z)$ such that
\begin{equation}\label{eq27}
z^n + z^{-n} = R_n (z+ z^{-1}) \, , \quad \forall z \in \C \setminus \{0\}
\, .
\end{equation}
The polynomial $R_n$ is monic with integer coefficients and has degree
$n$.
\end{prop}

\noindent
Clearly, $R_0 (z)\equiv 2$, $R_1(z) = z$, and for all $n\ge 2$ the
polynomial $R_n$ satisfies
$$
R_n(z) = z R_{n-1}(z) - R_{n-2} (z) \, , \quad \forall z\in \C \, .
$$
Letting $z = e^{\imath t}$ in (\ref{eq27}) yields $2 \cos
(nt) = R_n (2\cos t)$, and hence shows that $R_n (z)= 2 T_n (\frac12
z)$, where $T_n$ is the classical $n$-th Chebyshev polynomial of the
first kind. Also, $R_n (0) = 2 \cos \bigl( \frac12 n\pi \bigr)\in \{-2,0,2\}$, and for
every $n\ge 3$ the polynomial $R_n(z) - z^n + nz^{n-2}$ has degree
less than $n-2$.

To identify the minimal polynomial over $\Q$ of $\cos (\pi r)$, say, observe
that for every $n\ge 2$ and $z\in \C\setminus \{0\}$,
\begin{equation}\label{eq29}
\Phi_{2n}(z) = \sum\nolimits_{j=0}^{\varphi(2n)} a(j,2n)
z^{\varphi(2n) - j} = z^{\frac12 \varphi(2n)} P_n (z+ z^{-1}) \, ,
\end{equation}
with the polynomial $P_n = P_n(z)$ given by
$$
P_n(z) = \sum\nolimits_{j=0}^{\frac12 \varphi(2n) - 1} a(j,2n)
R_{\frac12 \varphi(2n) -j}(z) + a\left( {\textstyle \frac12} \varphi(2n) , 2n \right) \, ;
$$
in addition, define $P_1(z) = z+ 2$. With this, the degree of $P_n$
simply equals $p_n$, where
\begin{equation}\label{eq29a}
p_n = \left\{
\begin{array}{lll} 1 & & \mbox{\rm if } n = 1 \\[0.5mm]
\frac12 \varphi(2n) & & \mbox{\rm if } n\ge 2
\end{array}
\right\}
=
 \left\{
\begin{array}{lll} 1 & & \mbox{\rm if } n = 1 \, , \\
\varphi(n) & & \mbox{\rm if $n\ge 2$ is even}\, , \\[0.5mm]
\frac12 \varphi(n) & & \mbox{\rm if $n\ge 2$ is odd}\, . 
\end{array}
\right.
\end{equation}
For example,
$$
P_2(z)= z \, , \quad P_3(z) = z - 1 \, , \quad P_4(z) = z^2 - 2 \, ,
\quad
P_5(z) = z^2 - z -1 \,  .
$$
Clearly, each $P_n$ is monic with integer coefficients, and
(\ref{eq29}) implies that $P_n$ is irreducible over a field $K$ with
$\Q \subset K \subset \C$ whenever $\Phi_{2n}$ is irreducible over
$K$; in particular, $P_n$ is irreducible over $\Q$. Also, by
(\ref{eq29}) and Lemma \ref{lem25},
\begin{equation}\label{eq210}
|P_n(0)| = |\Phi_{2n}(\imath)| = \left\{
\begin{array}{cll}
2 & & \mbox{\rm if } n = 1 \, , \\
0 & & \mbox{\rm if } n = 2 \, , \\
p & & \mbox{\rm if $n = 2 p^j$ for some prime number $p$ and $j\in \N$} \, , \\[0.5mm]
1 & & \mbox{\rm otherwise}\, .
\end{array}
\right.
\end{equation}
The relevance of the polynomials $P_n$ for the present work comes from
the following simple consequence of (\ref{eq29}) which refines the
classical fact \cite[Thm.1]{Lehmer33}.

\begin{lem}\label{lem210}
Let $r\in \Q$. Then the polynomial $P_{N(r)}$ is the minimal polynomial over $\Q$ of the number
$2(-1)^{1+rN(r)} \cos(\pi r)$. In particular, the degree over $\Q$ of
$\cos (\pi r)$ is $p_{N(r)}$.
\end{lem}

\begin{proof}
Fix $r\in \Q$, and let $k= rN(r)\in \Z$ for convenience. Note that $k$
and $N(r)$ are coprime. If $N(r)=1$, or equivalently if $r\in \Z$,
then $2(-1)^{1+rN(r)} \cos (\pi r) = -2$ clearly solves
$P_1(z)=0$. Hence assume $N(r)\ge 2$ from now on.

If $k$ is odd then $2N(r)$ and $k$ are coprime, and (\ref{eq29})
implies
$$
P_{N(r)} \bigl( 2\cos (\pi r)\bigr) = \left( e^{\pi \imath r}
\right)^{-p_{N(r)}} \Phi_{2N(r)} \left( e^{2\pi \imath
    k/(2N(r))}\right) = 0 \, .
$$
If, on the other hand, $k$ is even then $N(r)$ is odd, and $2N(r)$ and
$k+N(r)$ are coprime. In this case, (\ref{eq29}) yields
\begin{align*}
P_{N(r)} \bigl(-2\cos(\pi r)\bigr)
& = \left( - e^{\pi \imath r }\right)^{-p_{N(r)}}  \Phi_{2N(r)} \left(
-e^{\pi \imath r} \right) \\
& = \left( - e^{\pi \imath r }\right)^{-p_{N(r)}} \Phi_{2N(r)} \left(
e^{2\pi \imath (k+N(r))/(2N(r))} 
\right) = 0 \, .
\end{align*}
In either case, therefore, $z=2(-1)^{1+k}\cos(\pi r)$ solves $P_{N(r)}(z)=0$.
\end{proof}

\begin{rem}
Fact \ref{fac1} is an immediate consequence of
Lemma \ref{lem210} since, as is easily checked, $p_n =1$ if and only
if $n\in \{ 1,2,3\}$. Note, however, that Fact \ref{fac1} can also be
established in an entirely elementary manner \cite{Jahnel10,
  Olmsted45}. As a simple corollary, observe that the number $\pi^{-1} \arccos \sqrt{r}$, with $r \in
\Q$ and $0\le r \le 1$, is rational if and only if $4r \in
\{0,1,2,3,4\}$; see \cite{Varona06}.
\end{rem}

For the trigonometric algebraic numbers $\sin (\pi r)$, a 
result completely analogous to Lemma \ref{lem210} holds. To state it, for every $n\in \N$ define the
polynomial $Q_n=Q_n(z)$ as
\begin{equation}\label{eq232}
Q_n(z) = \left\{
\begin{array}{lll}
P_{2n}(z) & & \mbox{\rm if $n$ is odd}\, , \\[0.5mm]
P_{n}(z) & & \mbox{\rm if $n\in 4\Z\, ,$}\\[0.5mm]
P_{\frac12 n}(z) & & \mbox{\rm if $n\in 2 + 4 \Z\, , $}\\
\end{array}
\right.
\end{equation}
and note that the degree of $Q_n$ simply equals $q_n$, where
$$
q_n = 
 \left\{
\begin{array}{lll} 1 & & \mbox{\rm if } n = 2 \, , \\[0.5mm]
\frac12 \varphi(n) & & \mbox{\rm if $n\in 2 + 4 \Z$ and $n\ge 6\, , $}\\[0.5mm]
\varphi(n) & & \mbox{\rm otherwise}\, . 
\end{array}
\right.
$$
Each $Q_n$ is monic with integer coefficients, and is irreducible
precisely if the corresponding $P_m$ according to (\ref{eq232}) is
irreducible. The following analogue of Lemma \ref{lem210} is easily
deduced from the latter; details are left to the reader.

\begin{lem}\label{lem211}
Let $r\in \Q$. Then the polynomial $Q_{N(r)}$ is the minimal polynomial over $\Q$ of the number $2(-1)^{k_r}
\sin(\pi r)$, where $k_r = 1 + \frac14
  N(r)|2r-1|$ if $N(r)\in 2 + 4 \Z$, and $k_r = 0$ otherwise. In
particular, the degree over $\Q$ of $\sin(\pi r)$ is $q_{N(r)}$.
\end{lem}

\section{Proof of the main result}\label{sec3}

By making use of the tools assembled above, it is now possible to present a

\begin{proof}[Proof of Theorem \ref{thm2}]
Fix $r_1,r_2\in \Q$ such that neither $r_1 - r_2$ nor $r_1 + r_2$ is an
integer, and for convenience let $N_j = N(r_j)$ for $j\in \{ 1,2\}$, as well as $n_j = p_{N_j}$ and $z_j = 2(-1)^{1 + r_j N_j} \cos
(\pi r_j)$; plainly, $1$, $\cos(\pi
r_1)$, and $\cos(\pi r_2)$ are $\Q$-independent if and only if $1$,
$z_1$, and $z_2$ are.

To see that (i)$\Rightarrow$(ii), simply note that $p_{N(r)}=1$, and
hence $\cos(\pi r)\in \Q$, whenever $N(r)\le 3$. Thus if $1$, $z_1$,
and $z_2$ are $\Q$-independent then necessarily $N_1, N_2\ge
4$. Also, from (\ref{eq01}) it is evident that $( N_1,N_2)\ne (5,5)$ in this case. 

It remains to establish the reverse implication
(ii)$\Rightarrow$(i). To this end, assume for the time being that
$n_1, n_2\ge 3$, or equivalently $N_1, N_2\ge 7$. Then (ii)
holds, $z_1$ and $z_2$ both are irrational, and the goal is to show
that $1$, $z_1$, and $z_2$ are $\Q$-independent. Assume, therefore, that
$rz_1 + sz_2 +t = 0$ with $r,s,t\in\Q$. If $r=0$ then $s=t=0$, so
assume further that $r\ne 0$, and w.l.o.g.\ let $r=-1$. Thus, with $s,t\in \Q$ and
$s\ne 0$,
\begin{equation}\label{eq30a}
z_1 = s z_2 + t \, .
\end{equation}
The proof will be complete, at least for the case of $N_1,
N_2\ge 7$, once it is shown that (\ref{eq30a}) always fails. This will
now be done by separately considering two cases.

\medskip

\noindent
{\bf Case I: $t=0$.} 

\noindent
Assume first that $t=0$ in (\ref{eq30a}). Then $z_1$ and $z_2$ have
the same degree over $\Q$, i.e., $n=n_1 = n_2 \ge 3$, as well as
minimal polynomials $P_{N_1}$ and $P_{N_2}$, respectively. From
$$
0 = s^{-n} P_{N_1} (z_1) = s^{-n} P_{N_1} (sz_2) = z_2^n +
\ldots + s^{-n} P_{N_1}(0) \, ,
$$
together with the uniqueness of the (monic) minimal polynomial $P_{N_2}$, it
follows that
\begin{equation}\label{eq31}
P_{N_1}(0) = s^n P_{N_2} (0) \, .
\end{equation}
Recall from (\ref{eq210}) that $|P_{N_j}(0)|$ equals $1$ or a prime
number. Thus, if $|P_{N_1}(0)|\ne |P_{N_2}(0)|$ then
(\ref{eq31}) is impossible for $s\in \Q$. If, on the other hand,
$|P_{N_1}(0)| = |P_{N_2}(0)|$ and (\ref{eq31}) does have a solution then $|s|=1$, which in turn implies that 
$$
\cos (\pi r_1) + \cos (\pi r_2) = 0 \quad\mbox{\rm or} \quad
\cos (\pi r_1) - \cos (\pi r_2) = 0 \, .
$$
In either case, at least one of the numbers $r_1 - r_2$ and $r_1 + r_2$ is an integer,
contradicting the standing assumption that none of them is.
In summary, (\ref{eq30a}) fails whenever $t=0$. In
particular, $z_1/z_2$ is irrational.

\medskip

\noindent
{\bf Case II: $t\ne 0$.}
 
\noindent
Assume from now on that (\ref{eq30a}) holds with $s,t\in \Q$ and
$st\ne 0$. Again, $z_1$ and $z_2$ have the same degree over $\Q$, thus
$n=n_1 = n_2 \ge 3$. For convenience, let
$$
(a_j, b_j ,c_j) = \bigl(
a(1, 2 N_j), a(2, 2 N_j),a(3, 2 N_j) \bigr) \, , \quad j\in \{ 1,2 \} \, ,
$$
and consequently
$$
P_{N_j} (z) = z^n + a_j z^{n-1} + (b_j - n) z^{n-2} + \bigl( c_j -
a_j (n-1) \bigr) z^{n-3} + U_j(z) \, , \quad j\in \{ 1,2\} \, ,
$$
where $U_j$ denotes an appropriate polynomial of degree less than
$n-3$. With (\ref{eq30a}), it follows that
\begin{align}\label{eq32}
0   = s^{-n} P_{N_1} (z_1) & = s^{-n} P_{N_1} (sz_2 + t) \\[1mm]
&= z_2^n + \widetilde{a}_2 z_2^{n-1} + \widetilde{b}_2 z_2^{n-2} +
\widetilde{c}_2 z_2^{n-1} + \widetilde{U}_2 (z_2) =: \widetilde{P}_{N_2} (z_2) \, , \nonumber
\end{align}
with a polynomial $\widetilde{U}_2$ of degree less than $n-3$, and coefficients
\begin{align*}
\widetilde{a}_2 & =\frac{nt+ a_1}{s} \, , \\[0.5mm]
\widetilde{b}_2 & = \frac{ n(n-1) t^2 + 2  a_1(n-1) t + 2(b_1 - n)}{2s^2} \, , \\[0.5mm]
\widetilde{c}_2 & = \frac{ n(n-1) (n-2) t^3 + 3 a_1 (n-1) (n-2)t^2
  + 6(b_1 - n) (n-2) t + 6c_1 - 6a_1 (n-1)}{6s^3} \, .
\end{align*}
Requiring that $\widetilde{P}_{N_2} = P_{N_2}$ yields
\begin{align}\label{eq33}
s a_2 & = nt + a_1  \, , \nonumber \\
2s^2 (b_2 - n)  & = n(n-1) t^2 + 2a_1 (n-1) t + 2 b_1 - 2n  \, , \\
6s^3 \bigl( c_2 - a_2 (n-1)\bigr) & = n(n-1)(n-2) t^3 + 3 a_1 (n-1)(n-2) t^2 + & \nonumber \\
& \quad + 6(b_1 - n) (n-2) t + 6 c_1 - 6 a_1(n-1) \, . \nonumber
\end{align}
Note that (\ref{eq33}) consists of three equations for the two (rational) numbers
$s$ and $t$. Quite plausibly, therefore, (\ref{eq33}) may be contradictory,
which in turn would cause (\ref{eq30a}) to fail also, just as
desired. It will now be shown that this is exactly what happens,
regardless of the actual values of $n\ge 3$ and the coefficient triples $(a_j, b_j,
c_j)$. In order to do so, it is convenient to distinguish three subcases,
depending on whether none, exactly one, or both of the integers $N_j$ are
squarefree. 

\smallskip

{\em Case IIa.} Assume first that neither $N_1$ nor $N_2$ is squarefree. Then,
by Lemma \ref{lem23}, $a_1 = a_2 = 0$, and the first equation in
(\ref{eq33}) reduces to $0= nt$, which contradicts the assumption
$t\ne 0$. Hence (\ref{eq33}) fails if neither $N_1$ nor
$N_2$ is squarefree.

\smallskip

{\em Case IIb.} Next assume that exactly one of the two integers $N_1$ and $N_2$
is squarefree; w.l.o.g.\ let $N_2$ be squarefree. (Otherwise interchange the roles of $z_1$ and $z_2$.) Hence $a_1 = 0$, and by
replacing $z_2$ with $-z_2$ if necessary, it can be assumed that
$a_2=1$ and consequently, by Lemma \ref{lem22}, the pair $(b_2,c_2)$
has exactly one of the following four values:
\begin{equation}\label{eq34}
(1,1)  , \: (1,0) , \: (0,0)  , \:  (0,-1) \, .
\end{equation}
In this case, the first equation in (\ref{eq33}) reads $s=nt$, and
the other two equations become
\begin{eqnarray}\label{eq35}
n(2n^2 + (1 - 2b_2) n - 1) t^2 - 2 (n-b_1) & = & 0 \, , \nonumber
\\[-3.5mm]
& \\ [-3.5mm]
n(6n^3 - (5+6c_2) n^2 - 3n + 2)t^3 - 6(n^2 -(2+b_1) n + 2b_1) t + 6c_1
& = & 0 \, . \nonumber
\end{eqnarray}
Note that $V_0(n;b_2) = 2n^2 + (1-2b_2)n-1 \ne 0$ for all $n\ge 3$ and $b_2
\in \{0,1\}$. It follows that
\begin{equation}\label{eq36}
t^2 = \frac{2}{n} \cdot \frac{n-b_1}{V_0(n;b_2)} \, ,
\end{equation}
and plugging this into the second equation in (\ref{eq35}) yields,
after a short calculation,
\begin{equation}\label{eq37}
t = - \frac{3c_1}{2} \cdot \frac{V_0(n;b_2)}{ V_1(n)} \, ,
\end{equation}
with the cubic polynomial $V_1$ given by
$$
V_1(z) = (2+ 3b_2 -
  3c_2)z^3 + (3 - 2b_1 - 6b_2 - 3b_1b_2 + 3b_1c_2)z^2 - (2+3b_1 -
  6b_1b_2) z + 2b_1 \, .
$$
Note that $c_1 \ne 0$ by (\ref{eq36}) and (\ref{eq37}), and hence $|c_1|=1$. Again, it is readily
confirmed that $V_1(n)\ne 0$ for all $n\ge 3$ and all relevant values
of $b_1$, $b_2$, and $c_2$. In order
for (\ref{eq36}) and (\ref{eq37}) to be compatible, the (seventh degree
polynomial) equation
\begin{equation}\label{eq38}
9 n V_0(n;b_2)^3 = 8(n-b_1) V_1(n)^2
\end{equation}
must be satisfied. It is now an elementary task to check that this is
not the case for any $n\ge 3$, any $b_1 \in \{-1,0,1\}$, and any pair $(b_2,c_2)$ from (\ref{eq34}).
For example, for $b_1= 1$ and $(b_2,c_2)=(1,1)$, the condition (\ref{eq38}) takes the form
\begin{align*}
0 & = 40n^7 + 84 n^6 -446 n^5 +347 n^4 +163 n^3 -211 n^2 -9n +32  \\
 & = (n-1)^3 (2n+1)^2 (10n^2 +41 n -32) \, ,
\end{align*}
which for $n\in \N$ only holds if $n=1$. The altogether eleven other
possibilities for $b_1$ and $(b_2,c_2)$ are
dealt with in a completely similar manner. In summary, (\ref{eq33})
fails if exactly one of the numbers $N_1$ and $N_2$ is
squarefree.

\smallskip

{\em Case IIc.} Finally, assume that both $N_1$ and $N_2$ are squarefree. In
this case, it can also be assumed that $a_1=a_2 = 1$, and the pairs
$(b_1,c_1)$ and $(b_2,c_2)$ each have exactly one of the four values
(\ref{eq34}). Now the first equation in (\ref{eq33}) reads $s = nt+1
$, and with this the two other equations reduce to
\begin{eqnarray}\label{eq38na}
(nt+1)^2 & = & \frac{V_0(n;b_1)}{V_0(n;b_2)} \, , \nonumber \\[-2.5mm]
& \\[-2.5mm]
(nt+1)^3 - 3(nt+1) \frac{V_2(n)}{V_4(n)} & = & 2 \frac{V_3(n)}{V_4(n)}
\, ,  \nonumber
\end{eqnarray}
where the polynomials $V_2$, $V_3$, and $V_4$ are given by
\begin{align*}
V_2(z) & = 2z^3 - (3+2b_1) z^2 - (3-4b_1) z + 2 \, , \\
V_3(z) & = (2+3b_1 - 3c_1) z^2 + (3 - 6b_1) z - 2 \, , \\
V_4(z) & = 6z^3 - (5+6c_2) z^2 - 3z + 2 \, .
\end{align*}
As before, $V_4(n)\ne 0$ for all $n\ge 3$ and $c_2 \in \{-1,0,1\}$.
If $b_1=b_2$ then the first equation in (\ref{eq38na}) yields $nt+1
\in \{-1,1\}$, and so $nt=-2$, since $nt=0$ would contradict the
assumption $t\ne 0$. The second equation
in (\ref{eq38na}) then becomes
$$
0 = V_4(n) + 2V_3(n) - 3V_2(n) = 2(4+6b_1 - 3c_1 - 3c_2)n^2 + 12
(1-2b_1)n - 8 \, ,
$$
which is readily confirmed to not have any integer solution $n\ge 3$
whenever $b_1 = b_2 \in \{0,1\}$ and $c_1,c_2 \in \{-1,0,1\}$. If, on
the other hand, $b_1 \ne b_2$ then in order for the two equations in (\ref{eq38na}) to be compatible, the
(tenth degree polynomial) equation
\begin{equation}\label{eq39}
V_0 (n;b_1) \bigl(
V_0(n;b_1) V_4(n) - 3 V_0(n;b_2) V_2(n)
\bigr)^2 = 4 V_0(n;b_2)^3 V_3(n)^2
\end{equation}
must be satisfied. Similarly to Case IIb, it is straightforward to check
that (\ref{eq39}) does not hold for any $n\ge 3$ and any two pairs
$(b_j,c_j)$ from (\ref{eq34}) with $b_1 \ne b_2$. For example, if 
$(b_1, c_1)=(0,-1)$ and $(b_2,c_2)=(1,1)$ then (\ref{eq39}) takes the form
\begin{align*}
0 & = 672n^{10} -48n^9-1752n^8-20n^7+1332n^6 -92 n^5 -444n^4  + 16n^3 +48n^2  \\
& = 4n^2 (n+1)^2 (2n+1)^2 (42n^4 -129 n^3 +141 n^2 -68n + 12)  \, ,
\end{align*}
which has no solution $n\in \N$. The altogether seven other
possibilities for $(b_1,c_1)$ and $(b_2,c_2)$ with $b_1 \ne b_2$ are
dealt with in a completely similar manner. As a consequence,
(\ref{eq33}) fails whenever $N_1$ and $N_2$ are both
squarefree. As explained earlier, this completes the proof of the
implication (ii)$\Rightarrow$(i) in the case of $N_1, N_2\ge 7$.

\medskip

It remains to consider those situations where $N_j\in \{4,5,6\}$
for at least one $j$. Hence assume w.l.o.g.\ that $N_1\in
\{4,5,6\}$, and thus $n_1 = 2$. Clearly, $1$, $\cos(\pi r_1)$, and
$\cos(\pi r_2)$ are $\Q$-independent unless $n_2 = 2$ as well. Thus both
$z_1$ and $z_2$ are roots of one of the irreducible polynomials
$$
P_4(z) = z^2 - 2 \, , \quad P_5(z) = z^2 - z - 1 \, , \quad P_6 (z) =
z^2 - 3 \, .
$$
If, for instance, $N_1=4$ then (\ref{eq30a}) implies that, in analogy to (\ref{eq32}),
\begin{equation}\label{eq310}
0 = s^{-2} P_4(z_1) = s^{-2} P_4 (sz_2 + t) = z_2^2 +  \frac{2t}{s}
z_2 + \frac{t^2 - 2}{s^2} =: \widetilde{P}_4 (z_2) \, .
\end{equation}
Note that $\widetilde{P}_4 \ne P_4$ because otherwise $(s,t)$ would
equal $(1, 0)$ or $(-1,0)$, and therefore, as seen earlier, one of the numbers $r_1-r_2$
and $r_1+r_2$ would be an integer; but also $\widetilde{P}_4\ne P_5$ because
otherwise $5s^2 = 8$, which is impossible for $s\in \Q$; and
$\widetilde{P}_4 \ne P_6$ because otherwise $3s^2 = 2$, which is likewise impossible. The assumption
$N_1=6$ leads to a similar string of contradictions. In summary, this shows
that (\ref{eq30a}) cannot hold whenever $N_1, N_2\in \{4,5,6\}$
but $( N_1, N_2)\ne (5,5)$. 
\end{proof}

\begin{rem}\label{rem32}
The special role played by the case of $N_1=N_2=5$ in the above
argument is highlighted by the fact that, in analogy to
(\ref{eq310}),
$$
s^{-2} P_5(sz+t) = z^2 + \frac{2t-1}{s} z + \frac{t^2 - t -1}{s^2}
=:\widetilde{P}_5(z) \, ,
$$
and $\widetilde{P}_5 = P_5$ for $(s,t)=(-1,1)$. This also explains the
validity of (\ref{eq01}).
\end{rem}

Upon close inspection, the argument given above can be seen to be
independent of the underlying field being $\Q$. The same reasoning
applies over larger fields, provided that $P_n$ remains
irreducible, and (\ref{eq31}) has no solution with $|s|\ne 1$. 
Theorem \ref{thm2} can thus be strengthened without additional effort.

\begin{theorem}\label{thm33}
Let $r_1,r_2\in \Q$ be such that neither $r_1 - r_2$ nor $r_1 + r_2$
is an integer, and assume the integer $d\ge 2$ is squarefree, with
$\mbox{\rm gcd}\bigl( d, N(r_j)\bigr)=1$ for $j\in \{1,2\}$. Then 
the following are equivalent:
\begin{enumerate}
\item The numbers $1$, $\cos(\pi r_1)$, and $\cos (\pi r_2)$ are
  $\Q\bigl( \sqrt{d}\bigr)$-independent;
\item $N(r_j)\ge 4$ for $j\in \{ 1,2\}$, and $\bigl( N(r_1), N(r_2)\bigr)\ne (5,5)$.
\end{enumerate}
\end{theorem}

\begin{proof}
Since $\Q\bigl( \sqrt{d}\bigr)$-independence implies
$\Q$-independence, the implication (i)$\Rightarrow$(ii) is obvious
from Theorem \ref{thm2}. To see the converse, observe that if
$d$ and $N(r)$ are coprime then $\widehat{d}\nmid 2 N(r)$, and hence 
$P_{N(r)}$, the minimal polynomial over
$\Q$ of $2(-1)^{1+rN(r)}\cos(\pi r)$, is irreducible over $\Q\bigl(
\sqrt{d}\bigr)$, as a consequence of (\ref{eq29}) and Lemma
\ref{lem26}. In particular, $\cos(\pi r)$ has the same degree
$p_{N(r)}$ over $\Q\bigl( \sqrt{d}\bigr)$ as it has over
$\Q$. Furthermore, notice that if $|P_{N(r_1)}(0)|\ne |P_{N(r_2)}(0)|$
then (\ref{eq31}) has no solution $s$ in $\Q\bigl( \sqrt{d}\bigr)$
since the degree over $\Q$ of $s$ is at least $3$. Thus when $\Q$ 
is replaced with $\Q\bigl( \sqrt{d}\bigr)$, the
proof of (ii)$\Rightarrow$(i) carries over verbatim from the proof of Theorem \ref{thm2}.
\end{proof}

To put Theorem \ref{thm33} into perspective, note that with $r_1 =
\frac18$ and $r_2 = \frac38$, the numbers $\cos(\pi r_1)$ and
$\cos(\pi r_2)$, while $\Q$-independent by Theorem \ref{thm2}, are $\Q\bigl( \sqrt{2}\bigr)$-dependent, since
$\cos(3\pi/8) = (\sqrt{2} -1) \cos(\pi/8)$.
This is consistent with the fact that $N(r_1) = N(r_2)=8$ is divisible
by $d=2$. Thus the implication (ii)$\Rightarrow$(i) in Theorem \ref{thm33} may
fail if $d$ and $N(r_j)$ have a common factor. However, this need not happen always, i.e., the numbers $1$,
$\cos (\pi r_1)$, and $\cos (\pi r_2)$ may well be $\Q\bigl(
\sqrt{d}\bigr)$-independent even in cases where $\mbox{\rm gcd}\bigl(d,
N(r_j)\bigr) \ne  1$. To see this, take for instance $r_1 = \frac1{16}$ and $r_2=
\frac7{16}$. Again $d=2$ divides $N(r_1)=N(r_2)=16$, and yet
the numbers $1$, $z_1 = 2 \cos (\pi/16 )$, and
$z_2 =2 \cos(7 \pi/16 )$ are linearly independent over every real
quadratic field. This follows
easily from the fact that the minimal polynomial over $\Q\bigl(
\sqrt{d}\bigr)$ of both $z_1$ and $z_2$ equals $P_{16}$ if $d\ne 2$, and
equals $z^4 - 4z^2 + 2 - \sqrt{2}$ if $d=2$. This example shows that
linear independence may be found even in cases
where $P_{N(r_j)}$ is not irreducible. In fact, the author conjectures
that in the context of Theorem \ref{thm33}, the numbers $1$, $\cos(\pi
r_1)$, and $\cos(\pi r_2)$ are linearly independent over {\em every\/}
real quadratic field whenever $N(r_j)\ge 7$ for $j\in \{1,2\}$, unless
$\bigl( N(r_1) , N(r_2)\bigr) \in \{(8,8),(10,10),(12,12) \}$.

It is natural to ask whether Theorem \ref{thm2} extends to {\em more
  than two\/} trigonometric numbers. For example,
while $1$, $\cos (\pi r_1)$, and $\cos(\pi r_2)$, with $N(r_1) =
N(r_2)=7$ and neither $r_1 - r_2$ nor $r_1  + r_2$ an integer, are
linearly independent over every real quadratic field, any four numbers $1$, $\cos
(\pi r_1)$, $\cos(\pi r_2)$, and $\cos(\pi r_3)$, with
$N(r_j)= 7$ for $j\in \{1,2,3\}$, necessarily are $\Q$-dependent because
\begin{equation}\label{eq320}
2\cos (\pi/7 ) -  2\cos (  2 \pi /7 ) + 2\cos (3   \pi/7 ) = 1\, . \
\end{equation}
Note that (\ref{eq320}), and also (\ref{eq01}), is a special case of
the fact that, for every odd $n\ge 3$,
\begin{equation}\label{eq321}
2 \sum\nolimits_{1\le j < \frac12 n : {\rm gcd}(j,n)=1} (-1)^j \cos
(\pi j/n) = \mu (n) \in \{-1,0,1\} \, .
\end{equation}
With (\ref{eq321}), which is a corollary to an
identity for Ramanujan sums \cite[\S 16.6]{Hardy}
but also follows easily from Remark \ref{rem24}(i) and Lemma \ref{lem210},
it is clear that trigonometric numbers with arbitrarily large
degree over $\Q$ may be $\Q$-dependent. There is, however, a certain
trade-off in that $\Q$-dependence between trigonometric numbers of
large degree only occurs between sufficiently many such
numbers. This suggests that a precise analogue of Theorem \ref{thm2}
for arbitrarily many trigonometric numbers $\cos (\pi r_j)$ may be
delicate to state (and prove). The reader may want to compare this to
the similar question regarding the $\Q$-independence of $1,
\sqrt{r_1}, \ldots , \sqrt{r_m}$ with arbitrary $m\in \N$ and positive $r_j\in \Q$, a question
for which it is straightforward to give a complete answer; cf.\ \cite{Newman60}.

\section{Classifying rational high school triangles}\label{sec4}

As an application of the results and tools assembled so far, this final
section provides a complete classification of rational high school
triangles in the form of Theorem \ref{thm3} and Figure \ref{fig1}. Recall that a planar
triangle $\Delta = (\delta_1, \delta_2, \delta_3)$, with angles
$0<\delta_1\le \delta_2 \le \delta_3<\pi $ and with $\ell_j$ denoting the
side-length vis-\`{a}-vis $\delta_j$, is {\em rational\/} if $\delta_j/\pi\in
\Q$ for all $j$, and is a {\em high school triangle\/} if
$\ell_1:\ell_2:\ell_3 = x_1:x_2:x_3$ with $x_j \in \Q \bigl(
\sqrt{d_j}\bigr)$ for some (not necessarily different) squarefree
integers $d_j \ge 2$. Note that being rational or high school both
are properties of a triangle's similarity type, not of an
individual realization. 

As will become clear shortly, the full proof of
Theorem \ref{thm3} requires a considerable amount of elementary
calculations which will be discussed fully for a few
representative cases only. Based on these, however, the interested
reader will have no difficulty filling in the details for the
remaining cases. Also, it should be emphasized that while the necessary calculations
may be sped up through the use of symbolic mathematical
software, they all can reasonably be carried out by hand as well.

Assume that $\Delta = (\delta_1, \delta_2,
\delta_3)$ is a rational high school triangle, and for convenience let
$r_j = \delta_j/\pi$. By the law of sines,
\begin{equation}\label{eq41}
\frac{\sin (\pi r_j)}{\sin (\pi r_k)} = \frac{\ell_j}{\ell_k} \in \Q
\bigl( \sqrt{d_j}, \sqrt{d_k} \bigr) \, .
\end{equation}
Number fields $K =  \Q \bigl( \sqrt{d_j}, \sqrt{d_k} \bigr)$ will henceforth be referred to as
{\em biquadratic}, a terminology advocated, e.g., by
\cite[V.2]{Froehlich}. Recall that, by the law of cosines,
$$
2\cos (\pi r_j) = \frac{\sum_{k\ne j} (\ell_k/\ell_j)^2
  -1}{\prod_{k\ne j} (\ell_k/\ell_j)} \, ,
$$
from which it is clear that the degree over $\Q$ of $\cos(\pi r_j)$
equals $1,2,4$, or $8$. With (\ref{eq29a}) and Lemma \ref{lem210}, it
is readily confirmed that $p_{N(r_j)}\in \{1,2,4,8\}$ if and only if
either $N(r_j)=1$, which is impossible since $r_j\not \in \Z$, or else
if $N(r_j)$ has one of the following 14 values:
\begin{equation}\label{eq42}
2, \: 3, \: 4, \: 5, \: 6, \: 8, \: 10, \: 12, \: 15, \: 16, \: 17, \:
20, \: 24, \: 30 \, . 
\end{equation}
Recall from Lemma \ref{lem211} that $Q_{N(r)}$ is the minimal
polynomial over $\Q$ of $2(-1)^{k_r}\sin (\pi r)$. For each value $N$
in (\ref{eq42}), Figure \ref{fig2} lists the polynomial $Q_N$
as well as its factorizations over biquadratic fields. In preparation
for the proof of Theorem \ref{thm3} given below, this
information will first be used to address the following question
prompted by (\ref{eq41}).

\begin{quest}\label{qes1}
Let $r_1,r_2\in \Q$ be such that neither $r_1 - r_2$ nor $r_1 + r_2$
is an integer, and assume $N(r_j)$ is contained in {\rm (\ref{eq42})} 
for $j\in \{1,2\}$. For which values of $\bigl( N(r_1), N(r_2)\bigr)$ does
the ratio $\sin (\pi r_1)/\sin (\pi r_2)$ belong to a biquadratic field?
\end{quest}

\begin{figure}[!t]
\begin{center}
\psfrag{t2}[]{\small $2$}
\psfrag{t2a}[l]{\footnotesize $z+2=Q_2(z)$}
\psfrag{t3}[]{\small $3$}
\psfrag{t3a1}[r]{\footnotesize $\sqrt{3} \not \in K$}
\psfrag{t3a2}[l]{\footnotesize $z^2 - 3= Q_3(z)$}
\psfrag{t3b1}[r]{\footnotesize $\sqrt{3}  \in K$}
\psfrag{t3b2}[l]{\footnotesize $z \pm  \sqrt{3}$}
\psfrag{t4}[]{\small $4$}
\psfrag{t4a1}[r]{\footnotesize $\sqrt{2} \not \in K$}
\psfrag{t4a2}[l]{\footnotesize $z^2 - 2=Q_4(z)$}
\psfrag{t4b1}[r]{\footnotesize $\sqrt{2}  \in K$}
\psfrag{t4b2}[l]{\footnotesize $z \pm  \sqrt{2}$}
\psfrag{t5}[]{\small $5$}
\psfrag{t5a1}[r]{\footnotesize $\sqrt{5} \not \in K$}
\psfrag{t5a2}[l]{\footnotesize $z^4 - 5z^2 + 5=Q_5(z)$}
\psfrag{t5b1}[r]{\footnotesize $\sqrt{5}  \in K$}
\psfrag{t5b2}[l]{\footnotesize $z^2 -\frac12 \bigl(5 \pm \sqrt{5}\bigr)$}
\psfrag{t6}[]{\small $6$}
\psfrag{t6a}[l]{\footnotesize $z-1=Q_6(z)$}
\psfrag{t8}[]{\small $8$}
\psfrag{t8a1}[r]{\footnotesize $\sqrt{2} \not \in K$}
\psfrag{t8a2}[l]{\footnotesize $z^4 - 4z^2 + 2=Q_8(z)$}
\psfrag{t8b1}[r]{\footnotesize $\sqrt{2}  \in K$}
\psfrag{t8b2}[l]{\footnotesize $z^2 - 2 \pm  \sqrt{2}$}
\psfrag{t10}[]{\small $10$}
\psfrag{t10a1}[r]{\footnotesize $\sqrt{5} \not \in K$}
\psfrag{t10a2}[l]{\footnotesize $z^2 - z -1=Q_{10}(z)$}
\psfrag{t10b1}[r]{\footnotesize $\sqrt{5}  \in K$}
\psfrag{t10b2}[l]{\footnotesize $z -  \frac12 \bigl(1\pm \sqrt{5}\bigr)$}
\psfrag{t12}[]{\small $12$}
\psfrag{t12a1}[r]{\footnotesize $\bigl\{ \sqrt{2}, \sqrt{3},\sqrt{6}\bigr\} \cap  K = \varnothing $}
\psfrag{t12a2}[l]{\footnotesize $z^4 - 4z^2 + 1=Q_{12}(z)$}
\psfrag{t12b1}[r]{\footnotesize $\bigl\{ \sqrt{2}, \sqrt{3},\sqrt{6}\bigr\} \cap  K = \bigl\{\sqrt{2}\bigr\}$}
\psfrag{t12b2}[l]{\footnotesize $z^2 \pm z\sqrt{2} -1$}
\psfrag{t12c1}[r]{\footnotesize $\bigl\{ \sqrt{2}, \sqrt{3},\sqrt{6}\bigr\} \cap  K = \bigl\{\sqrt{3} \bigr\}$}
\psfrag{t12c2}[l]{\footnotesize $z^2 - 2 \pm\sqrt{3}$}
\psfrag{t12d1}[r]{\footnotesize $\bigl\{ \sqrt{2}, \sqrt{3},\sqrt{6}\bigr\} \cap  K = \bigl\{\sqrt{6}\bigr\}$}
\psfrag{t12d2}[l]{\footnotesize $z^2 \pm z\sqrt{6} +1$}
\psfrag{t12e1}[r]{\footnotesize $K = \Q \bigl(\sqrt{2}, \sqrt{3}\bigr)$}
\psfrag{t12e2}[l]{\footnotesize $z^* - \frac12 \bigl(\sqrt{6} \pm \sqrt{2}\bigr)$}
\psfrag{t15}[]{\small $15$}
\psfrag{t15a1}[r]{\footnotesize $\bigl\{ \sqrt{3}, \sqrt{5},\sqrt{15}\bigr\} \cap  K = \varnothing $}
\psfrag{t15a2}[l]{\footnotesize $z^8 - 7z^6 + 14z^4 - 8z^2 + 1=Q_{15}(z)$}
\psfrag{t15b1}[r]{\footnotesize $\bigl\{ \sqrt{3}, \sqrt{5},\sqrt{15}\bigr\} \cap  K = \bigl\{\sqrt{3}\bigr\}$}
\psfrag{t15b2}[l]{\footnotesize $z^4 \pm z^3\sqrt{3} - 2z^2 \mp 2z \sqrt{3} -1$}
\psfrag{t15c1}[r]{\footnotesize $\bigl\{ \sqrt{3}, \sqrt{5},\sqrt{15}\bigr\} \cap  K = \bigl\{\sqrt{5} \bigr\}$}
\psfrag{t15c2}[l]{\footnotesize $z^4 -\frac12 \bigl(7\pm \sqrt{5}\bigr)z^2 + \frac12
  \bigl(3\pm \sqrt{5}\bigr)$}
\psfrag{t15d1}[r]{\footnotesize $\bigl\{ \sqrt{3}, \sqrt{5},\sqrt{15}\bigr\} \cap  K = \bigl\{\sqrt{15}\bigr\}$}
\psfrag{t15d2}[l]{\footnotesize $z^4 \pm z^3 \sqrt{15} + 4z^2 -1$}
\psfrag{t15e1}[r]{\footnotesize $K = \Q \bigl(\sqrt{3}, \sqrt{5}\bigr)$}
\psfrag{t15e2}[l]{\footnotesize $z^2 + \frac12 \bigl(\sqrt{15} \pm \sqrt{3}\bigr) z^* +
  \frac12 \bigl(1\pm \sqrt{5}\bigr)$}
\psfrag{t16}[]{\small $16$}
\psfrag{t16a1}[r]{\footnotesize $\sqrt{2} \not \in K$}
\psfrag{t16a2}[l]{\footnotesize $z^8 - 8z^6 + 20z^4 - 16z^2 + 2=Q_{16}(z)$}
\psfrag{t16b1}[r]{\footnotesize $\sqrt{2}  \in K$}
\psfrag{t16b2}[l]{\footnotesize $z^4 - 4z^2 + 2\pm \sqrt{2}$}
\psfrag{t17}[]{\small $17$}
\psfrag{t17a1}[r]{\footnotesize $\sqrt{17} \not \in K$}
\psfrag{t17a2}[l]{\footnotesize $z^{16} -17z^{14}+119z^{12} -442
  z^{10} +935z^8-1122z^6+714z^4-204z^2+17= Q_{17}(z)$}
\psfrag{t17b1}[r]{\footnotesize $\sqrt{17}  \in K$}
\psfrag{t17b2}[l]{\footnotesize $z^8 - \frac12 \bigl (17\pm\sqrt{17}\bigr) z^6
  +\frac12 \bigl (51 \pm 7 \sqrt{17}\bigr)z^4 - \bigl (34\pm7\sqrt{17}\bigr)z^2 + 17\pm 4 \sqrt{17}$}
\psfrag{t20}[]{\small $20$}
\psfrag{t20a1}[r]{\footnotesize $\bigl\{ \sqrt{2}, \sqrt{5},\sqrt{10}\bigr\} \cap  K = \varnothing $}
\psfrag{t20a2}[l]{\footnotesize $z^8 - 8z^6 + 19z^4 - 12z^2 + 1=Q_{20}(z)$}
\psfrag{t20b1}[r]{\footnotesize $\bigl\{ \sqrt{2}, \sqrt{5},\sqrt{10}\bigr\} \cap  K = \bigl\{\sqrt{2}\bigr\}$}
\psfrag{t20b2}[l]{\footnotesize $z^4 \pm z^3\sqrt{2} - 3z^2 \mp 3z
  \sqrt{2} - 1$}
\psfrag{t20c1}[r]{\footnotesize $\bigl\{ \sqrt{2}, \sqrt{5},\sqrt{10}\bigr\} \cap  K = \bigl\{\sqrt{5} \bigr\}$}
\psfrag{t20c2}[l]{\footnotesize $z^4 -4 z^2 + \frac12   \bigl (3\pm \sqrt{5}\bigr)$}
\psfrag{t20d1}[r]{\footnotesize $\bigl\{ \sqrt{2}, \sqrt{5},\sqrt{10}\bigr\} \cap  K = \bigl\{\sqrt{10}\bigr\}$}
\psfrag{t20d2}[l]{\footnotesize $z^4 \pm z^3 \sqrt{10} + z^2 \mp z
  \sqrt{10} -1$}
\psfrag{t20e1}[r]{\footnotesize $K = \Q \bigl(\sqrt{2}, \sqrt{5}\bigr)$}
\psfrag{t20e2}[l]{\footnotesize $z^2 +\frac12 \bigl (\sqrt{10} \pm \sqrt{2}\bigr) z^* -
  \frac12 \bigl(1\mp \sqrt{5}\bigr)$}
\psfrag{t24}[]{\small $24$}
\psfrag{t24a1}[r]{\footnotesize $\bigl\{ \sqrt{2}, \sqrt{3},\sqrt{6}\bigr\} \cap  K = \varnothing $}
\psfrag{t24a2}[l]{\footnotesize $z^8 - 8z^6 +20z^4 -16z^2 + 1=Q_{24}(z)$}
\psfrag{t24b1}[r]{\footnotesize $\bigl\{ \sqrt{2}, \sqrt{3},\sqrt{6}\bigr\} \cap  K = \bigl\{\sqrt{2}\bigr\}$}
\psfrag{t24b2}[l]{\footnotesize $z^4 - \bigl (4\pm \sqrt{2}\bigr)z^2 + 3 \pm 2 \sqrt{2}$}
\psfrag{t24c1}[r]{\footnotesize $\bigl\{ \sqrt{2}, \sqrt{3},\sqrt{6}\bigr\} \cap  K = \bigl\{\sqrt{3} \bigr\}$}
\psfrag{t24c2}[l]{\footnotesize $z^4 - 4z^2 + 2 \pm \sqrt{3}$}
\psfrag{t24d1}[r]{\footnotesize $\bigl\{ \sqrt{2}, \sqrt{3},\sqrt{6}\bigr\} \cap  K = \bigl\{\sqrt{6}\bigr\}$}
\psfrag{t24d2}[l]{\footnotesize $z^4 - \bigl (4\pm\sqrt{6}\bigr) z^2 + 5 \pm 2 \sqrt{6}$}
\psfrag{t24e1}[r]{\footnotesize $K = \Q \bigl(\sqrt{2}, \sqrt{3}\bigr)$}
\psfrag{t24e2}[l]{\footnotesize $z^2 -\frac12 \bigl(4 \pm \sqrt{6} \pm
  \sqrt{2}\bigr)$ and $z^2 -\frac12 \bigl(4 \pm \sqrt{6} \mp
  \sqrt{2}\bigr)$}
\psfrag{t30}[]{\small $30$}
\psfrag{t30a1}[r]{\footnotesize $\sqrt{5} \not \in K$}
\psfrag{t30a2}[l]{\footnotesize $z^4 + z^3 - 4z^2 - 4z + 1=Q_{30}(z)$}
\psfrag{t30b1}[r]{\footnotesize $\sqrt{5}  \in K$}
\psfrag{t30b2}[l]{\footnotesize $z^2 +\frac12 \bigl (1\pm \sqrt{5}\bigr) z -
  \frac12 \bigl(3\mp \sqrt{5}\bigr)$}
\includegraphics{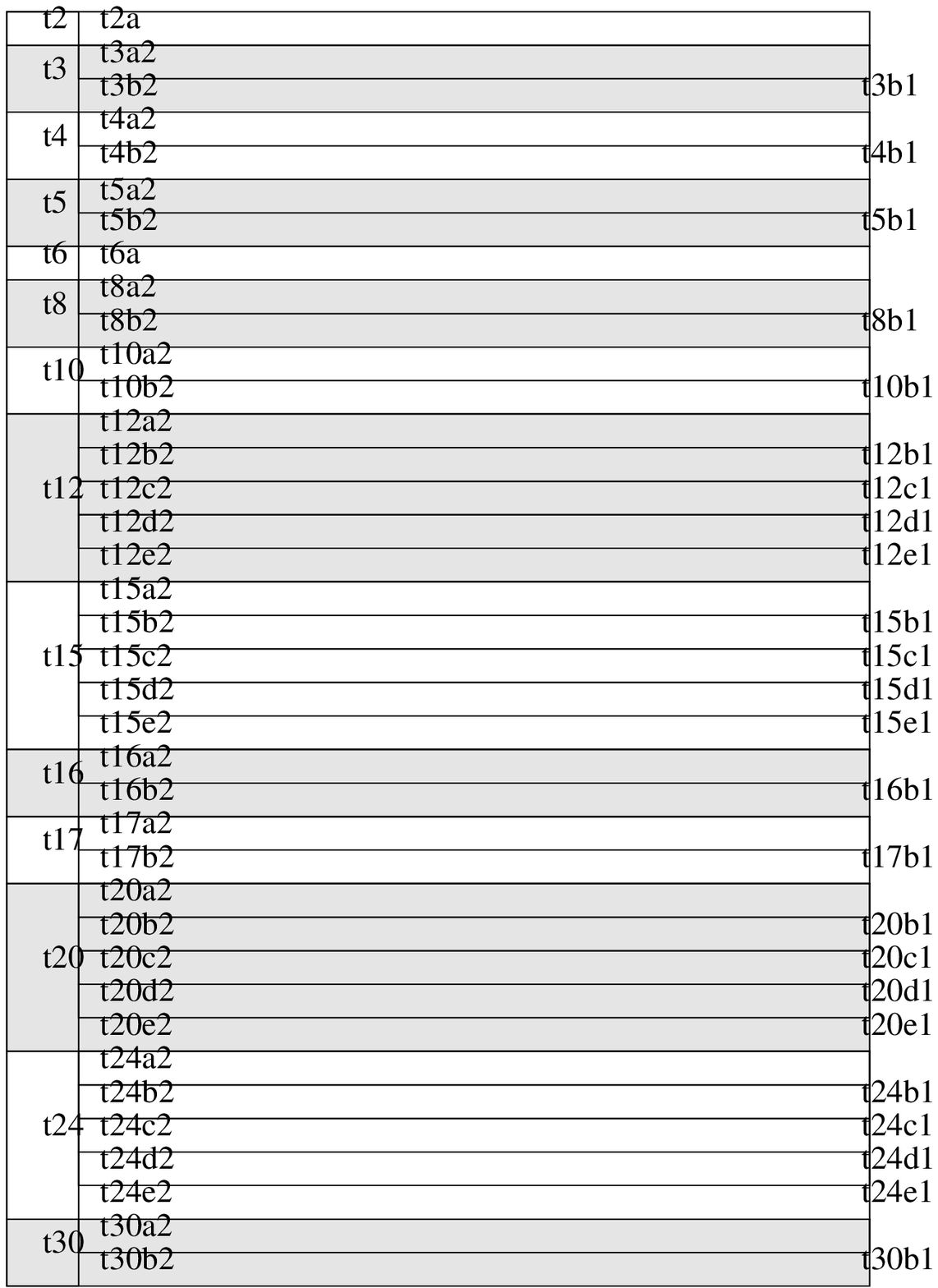}
\caption{The irreducible factors of $Q_N$ for all $N$ in (\ref{eq42}) over
  biquadratic fields $K$, determined by means of Lemmas \ref{lem26} and
  \ref{lem211}, in combination with (\ref{eq29}) and
  (\ref{eq232}). Expressions containing $\pm$ or $\mp$ are to be
  read as two separate expressions, containing only the upper and only
the lower signs, respectively; similarly, read $z^*$ as both $z^*=z$ and
$z^*=-z$.}\label{fig2}
\end{center}
\end{figure}

To address Question \ref{qes1}, fix $r_1,r_2\in \Q$ such that neither
$r_1-r_2$ nor $r_1 + r_2$ is an integer, and assume $N_j = N(r_j)$ is contained in
(\ref{eq42}) for $j\in \{1,2\}$. For convenience, let $w_j = 2 (-1)^{k_{r_j}}\sin (\pi
r_j)$; note that $w_1w_2 \ne 0$ and $(w_1/w_2)^2\ne 1$. 
Throughout, denote by $K$ any (real) biquadratic
field, further to be specified as appropriate.  Question \ref{qes1} is
well-posed in the sense that if $\sin (\pi r_1)/\sin (\pi r_2)$ is an
element of $K$ then so is $\sin (\pi s)/\sin (\pi t)$, provided that $\bigl( N(s),
N(t)\bigr)=(N_1, N_2)$. In other words, whether or not $\sin
(\pi r_1)/\sin (\pi r_2)\in K$ depends only on the pair
$(N_1,N_2)$. Also, observe that $\sin (\pi r_1)/\sin (\pi r_2)\in K$,
or equivalently $w_1/w_2\in K$, implies that the algebraic numbers
$w_1$ and $w_2$ have the same degree over $K$, that is, $\deg_K w_1 =
\deg_K w_2$. Failure of the latter equality, henceforth referred to informally
as {\em degree mismatch}, constitutes an easy-to-check sufficient condition for
$w_1/w_2\not \in K$.

In order to answer Question \ref{qes1}, due to symmetry a total of $105$ pairs $(N_1,N_2)$ have to be analysed. As
indicated earlier, the case-by-case analysis of these pairs requires
numerous near-identical calculations and arguments of an elementary nature. For the
reader's convenience only three values in (\ref{eq42}) are treated in
detail here. Put together, these three cases cover $39$ relevant pairs,
and the calculations presented are fully representative of the
ones left to the reader. Assume throughout that $u= w_1/w_2\in K$ with
$u\ne 0$ and $u^2 \ne 1$.

\medskip

\noindent
{\bf Case I: $N_1 = 2$.} 

\noindent
Plainly, $\deg_K w_2 = 1$, which, by Figure
\ref{fig2}, is the case if and only if $N_2 \in
\{2,3,4,6,10,12\}$. Conversely, the latter implies $w_1/w_2 \in K$;
see also Figure \ref{fig3}.

\medskip

\noindent
{\bf Case II: $N_1 = 5$.} 

\noindent
As seen in Case I, $(N_1,N_2)\ne (5,2)$. If $(N_1,N_2)=(5,3)$ then necessarily
$\sqrt{3}\not \in K$ but $\sqrt{5} \in K$, and hence $w_1^2 = \frac12\bigl(5
\pm \sqrt{5}\bigr)$ and $w_2^2 = 3$. (Here and throughout, expressions containing $\pm$ or $\mp$ are to be
read as two separate expressions, containing only the upper and
only the lower signs, respectively; cf.\ Figure \ref{fig2}.)
However, the ratio $(6w_1/w_2)^2 = 30
\pm  6 \sqrt{5}$ is easily seen to {\em not\/} be a square in $K$: If
$$
30 \pm  6 \sqrt{5} = \bigl(s_1 + s_2 \sqrt{5} + s_3 \sqrt{d} + s_4
\sqrt{5d} \bigr)^2
$$
for some $s_1, \ldots , s_4 \in \Q$  and squarefree integer $d\ge 2$
with $d\ne 5$, then
$$
s_1^2 + 5s_2^2 + ds_3^2 + 5d s_4^2 = 30 \, , \quad
s_1 s_2 + d s_3 s_4 = \pm 3 \, , \quad
s_1 s_3 + 5 s_2 s_4 = 0 \, , \quad
s_1 s_4 + s_2 s_3 = 0 \, ,
$$
and consequently either $s_1 = 0$, which in turn implies $s_2=0$ and
$ds_3^2 = 15 \pm  6
\sqrt{5}$, or else $s_1 \ne 0$ and hence $s_1^2 = 5s_2^2$ (if $s_3 \ne 0$) or $s_1^2 = 15
\pm  6 \sqrt{5}$ (if $s_3 = 0$). Since each alternative contradicts
the fact that $s_1, \ldots , s_4$ are rational numbers, it follows
that $(N_1, N_2)\ne (5,3)$. A completely analogous argument shows that
$(N_1,N_2)\ne (5,4)$.

If $N_1=N_2=5$ then either the ratio $(w_1/w_2)^2$ or its reciprocal equals
$$
\frac{5+\sqrt{5}}{5 - \sqrt{5}} = \left( \frac{1+\sqrt{5}}{2}
\right)^2 \, ,
$$
and consequently $w_1/w_2 \in \Q \bigl( \sqrt{5} \bigr)$.

Due to degree mismatch, clearly $(N_1,N_2)\ne (5,6)$. Next suppose
that $(N_1, N_2)=(5,8)$, and therefore, by Figure \ref{fig2},
\begin{equation}\label{eq4p10}
w_1^4 - 5 w_1^2 + 5 = 0 \quad \mbox{\rm and} \quad
w_2^4 - 4 w_2^2 + 2 = 0 \, .
\end{equation}
Deduce from (\ref{eq4p10}) via a short
calculation that 
\begin{equation}\label{eq4p11}
4u^8 - 40 u^6 + 110 u^4 - 100 u^2 + 25 = 0 \, .
\end{equation}
The polynomial in (\ref{eq4p11}) is irreducible over $\Q$, and so
$\deg_{\Q}(w_1/w_2)=8$, which contradicts the assumption of $w_1/w_2$ being
an element of a biquadratic field. Hence $(N_1,N_2)\ne (5,8)$. An essentially identical argument yields
$(N_1, N_2)\ne (5,12)$.

The possibility of $(N_1, N_2)=(5,10)$ is immediately ruled out by
degree mismatch since, by Figure \ref{fig2}, $\deg_K w_1 \ne \deg_K w_2$,
regardless of whether $\sqrt{5}\in K$ or $\sqrt{5}\not \in K$.

Next deduce from Figure \ref{fig2} that if $(N_1,N_2)=(5,15)$ then either
$\sqrt{5} \not \in K$ and $K$ contains exactly one of the two numbers
$\sqrt{3}$ and $\sqrt{15}$, or else $K = \Q\bigl( \sqrt{3}, \sqrt{5}\bigr)$. In
the first case, $w_1^4 - 5w_1^2 + 5 = 0$, while
$$
w_2^4 \pm w_2^3 \sqrt{3} - 2 w_2^2 \mp 2 w_2 \sqrt{3} - 1 = 0 \:\:\:
\bigl(   \sqrt{3} \in K\bigr ) \quad
\mbox{\rm or}
\quad
w_2^4 \pm w_2^3 \sqrt{15} + 4 w_2^2 - 1 = 0 \:\:\: \bigl( \sqrt{15} \in K \bigr) \, .
$$
Either alternative immediately yields the
contradiction $u=0$. In the second case,
$$
w_1^2 = {\textstyle \frac{1}{2}} \bigl(5\pm \sqrt{5}\bigr)  \quad
\mbox{\rm and} \quad
w_2^2 +  {\textstyle \frac{1}{2}} \bigl( \sqrt{15} \pm \sqrt{3} \bigr)
w_2^* +  {\textstyle \frac{1}{2}} \bigl(1\pm \sqrt{5}\bigr) = 0 \, ,
$$
where (as in Figure \ref{fig2}) the symbol $w_2^*$ is to be read as both
$w_2$ and $-w_2$. Again, $u=w_1/w_2 \in \Q \bigl( \sqrt{3},
\sqrt{5}\bigr)$ implies $u=0$. In summary, $(N_1,N_2)\ne (5,15)$. Identical
reasoning shows that $(N_1,N_2)\ne (5,20)$ and $(N_1, N_2)\ne (5,30)$.

Observe that $(N_1, N_2)=(5,16)$ is possible only if $\sqrt{2}\in K$
but $\sqrt{5}\not \in K$. However, in this case
$$
w_1^4 - 5w_1^2 + 5 = 0  \quad \mbox{\rm and} \quad 
w_2^4 - 4 w_2^2 + 2 \pm \sqrt{2} = 0 \, ,
$$
which entails the two contradictory conditions
$4u^2 = 5$ and $\bigl( 2\pm \sqrt{2}\bigr)u^4 = 5$ for $u\in K$. Thus $(N_1,N_2)\ne (5,16)$, and essentially
the same argument shows that $(N_1, N_2)\ne (5,24)$.
Finally, note that $(N_1, N_2)\ne (5,17)$ due to degree mismatch.
To summarize, $(N_1,N_2)=(5,N)$ with $N$ contained in (\ref{eq42}) is
possible if and only if $N=5$; see also Figure \ref{fig3}.

\medskip

\noindent
{\bf Case III: $N_1 = 8$.} 

\noindent
As demonstrated in Case I, $(N_1,N_2)\ne (8,2)$. Also, $(N_1,N_2)=(8,3)$ is possible only if
$\sqrt{2}\in K$ but $\sqrt{3}\not \in K$. In this case, however, just
as for the pair $(5,3)$ in Case II, the ratio $(3w_1/w_2)^2 = 6 \pm 3 \sqrt{2}$ is easily
seen to not be a square in $K$, and so $(N_1, N_2)\ne (8,3)$. The
possibility of $(N_1, N_2)\in \{(8,4), (8,6), (8,16), (8,17)\}$ is ruled out
by degree mismatch, while $(N_1, N_2)\ne (8,5)$ by Case II.

Next assume $(N_1,N_2)=(8,8)$. Then either $(w_1/w_2)^2$ or its
reciprocal equals
$$
\frac{2+\sqrt{2}}{2 - \sqrt{2}} = \bigl( 1+\sqrt{2} \bigr)^2 \, ,
$$
which in turn shows that $w_1/w_2 \in \Q \bigl( \sqrt{2} \bigr)$.

If $(N_1, N_2)=(8,10)$ then $\sqrt{2}\in K$ but $\sqrt{5}\not \in
K$. In this case,
$$
w_1^2 - 2\pm \sqrt{2}  = 0 \quad \mbox{\rm and} \quad
w_2^2 - w_2^* - 1 = 0 \, ,
$$
which implies $u=0$, a contradiction. Identical reasoning
yields $(N_1,N_2)\ne (8,12)$ and $(N_1, N_2)\ne (8,30)$ as well.

From Figure \ref{fig2}, it can be seen that $(N_1, N_2)=(8,15)$ is
possible only if $\sqrt{2}\not \in K$ and $K$ contains exactly one of
the three numbers $\sqrt{3}$, $\sqrt{5}$, and $\sqrt{15}$. In this case,
$w_1^4 - 4 w_1^2 + 2 = 0$, and the usual contradiction $u=0$ follows
in case $\bigl\{
\sqrt{3}, \sqrt{15}\bigr\}\cap K \ne \varnothing$. On the other hand,
if $\sqrt{5} \in K$ then
$$
w_2^4 - {\textstyle \frac12} \bigl(7 \pm \sqrt{5} \bigr) w_2^2 + {\textstyle
  \frac12} \bigl(3 \pm \sqrt{5}\bigr) = 0
$$
entails the contradictory conditions $ \bigl( 7\pm \sqrt{15}\bigr)u^2
= 8$ and $ \bigl( 3 \pm \sqrt{5}\bigr)u^4 = 4$ for $u\in K$. In summary, $(N_1, N_2)\ne (8,15)$, and
similarly $(N_1, N_2)\ne (8,20)$.

Finally, assume that $(N_1, N_2)=(8,24)$. This may be the case only if
either $\sqrt{2}\not \in K$ and $K$ contains exactly on of the two
numbers $\sqrt{3}$ and $\sqrt{6}$, or else if $K= \Q \bigl( \sqrt{2},
\sqrt{3}\bigr)$. In the first case, $w_1^4 - 4w_1^2 + 2 = 0$, while
$$
w_2^4 - 4 w_2^2 + 2 \pm \sqrt{3} = 0 \:\:\:  \bigl( 
\sqrt{3}\in K \bigr) \quad
\mbox{\rm or} \quad
w_2^4 - \bigl(  4\pm \sqrt{6}  \bigr)w_2^2 + 5 \pm 2\sqrt{6} = 0 \:\:\:
\bigl(  \sqrt{6}\in K \bigr) \, .
$$
As before, this yields contradictory conditions for $u\in K$.
In the second case, $w_1^2 = 2\pm \sqrt{2}$ while
$$
w_2^2 = {\textstyle \frac12} \bigl( 4 \pm \sqrt{6} \pm \sqrt{2} \bigr) 
\quad \mbox{\rm or} \quad
w_2^2 = {\textstyle \frac12} \bigl( 4 \pm \sqrt{6} \mp \sqrt{2} \bigr)
\, .
$$
Thus the ratio $(w_1/w_2)^2$ may, for instance, have the value
$$
\frac{4 + 2 \sqrt{2}}{4 +\sqrt{6} +\sqrt{2} } =
3 - \sqrt{2} + \sqrt{3} - \sqrt{6}
 = \left( \frac{2-\sqrt{2} -\sqrt{6}}{2}\right)^2 \, .
$$
Hence $w_1/w_2 \in K =\Q \bigl( \sqrt{2}, \sqrt{3}\bigr)$.
In summary, therefore, $(N_1,N_2)=(8,N)$ with $N$ contained in
(\ref{eq42}) is possible if and only if $N\in \{8,24\}$; see also
Figure \ref{fig3}.

\medskip

Figure \ref{fig3} displays graphically the results of the above
analysis, as well as of all the cases
left to the reader. It thus completely answers Question \ref{qes1}. Making use of
Figures \ref{fig2} and \ref{fig3}, it is now straightforward to to carry out the 

\begin{figure}[ht]
\begin{center}
\psfrag{tn1}[]{$N(r_1)$}
\psfrag{tn2}[]{$N(r_2)$}
\psfrag{t2}[]{\small $2$}
\psfrag{t3}[]{\small $3$}
\psfrag{t4}[]{\small $4$}
\psfrag{t5}[]{\small $5$}
\psfrag{t6}[]{\small $6$}
\psfrag{t8}[]{\small $8$}
\psfrag{t10}[]{\small $10$}
\psfrag{t12}[]{\small $12$}
\psfrag{t15}[]{\small $15$}
\psfrag{t16}[]{\small $16$}
\psfrag{t17}[]{\small $17$}
\psfrag{t20}[]{\small $20$}
\psfrag{t24}[]{\small $24$}
\psfrag{t30}[]{\small $30$}
\includegraphics{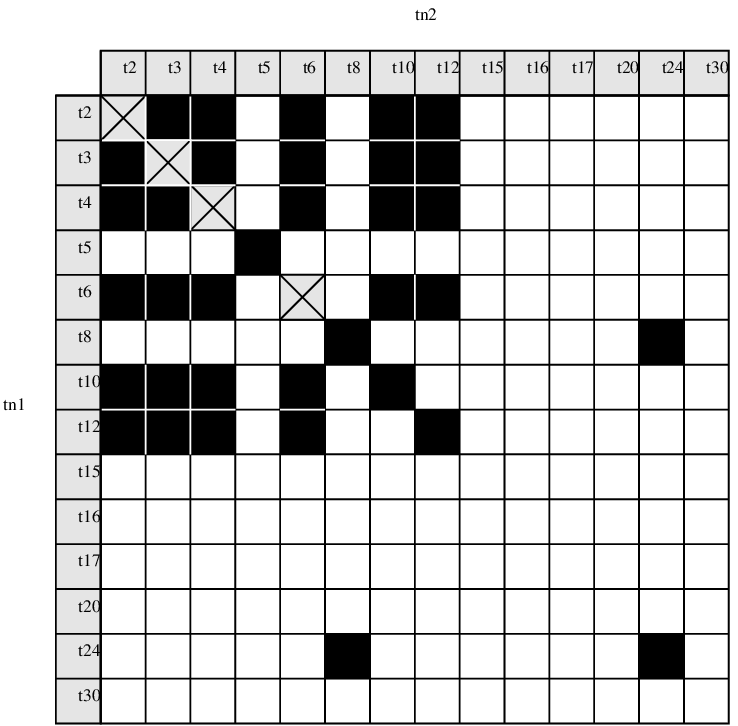}
\caption{For $r_1,r_2 \in \Q$ such that neither $r_1 -r_2$ nor
  $r_1+r_2$ is an integer, and for the values of $N(r_j)$ allowed by (\ref{eq42}), the ratio $\sin (\pi r_1)/\sin (\pi r_2)$ is contained in a
  biquadratic number field if and only if the pair $\bigl( N(r_1),
  N(r_2)\bigr)$ corresponds to a black box. Crossed-out grey boxes indicate
pairs that do not occur under the stated assumptions on $r_1,r_2$.}\label{fig3}
\end{center}
\end{figure}

\begin{proof}[Proof of Theorem \ref{thm3}]
Let the triangle $\Delta = (\delta_1, \delta_2, \delta_3)$ be
rational, and recall that $r_j = \delta_j/\pi$ for $j\in \{1,2,3\}$.
To establish the implication (i)$\Rightarrow$(ii), assume $\Delta$ is
a high school triangle. For convenience, the following argument is
split into two cases.

\medskip

\noindent
{\bf Case I:} $r_1=r_2$ or $r_2=r_3$, i.e., $\Delta$ is isosceles{\bf .}

\noindent
Assume first that $r_1 = r_2 =r$ for some $r\in \Q$ with
$r>0$. Since $r_3 = 1 - 2r \ge r$, necessarily $r\le \frac13$. For
$r=\frac13$ clearly $\Delta$ is equilateral, i.e., $\Delta = \pi
(1,1,1)/3$ and $N(\Delta)=3$. If $r<\frac13$ then neither
$r-(1-2r)=3r-1$ nor $r+(1-2r)=1-r$ is an integer, and consequently the pair $\bigl( N(r), N(1-2r)\bigr)$ must be admissible by
Figure \ref{fig3}. Necessarily, therefore,
$$
N(r) \in \{2,3,4,5,6,8,10,12,24\} \, .
$$
Clearly, $N(r)\in \{2,3\}$ is impossible since $0<r < \frac13$. Moreover, if $N(r)\in \{8,10,24\}$ then $N(1-2r)\mid \frac12 N(r)$,
which is impossible by Figure \ref{fig3}. This only leaves the
possibility of $N(r)\in \{4,5,6,12\}$, and it is straightforward to
check that the latter allows for exactly four (isosceles) triangles, namely
$$
\Delta = \pi (1,1,2)/4 \, , \quad
\Delta = \pi (1,1,3)/5 \, , \quad
\Delta = \pi (1,1,4)/6 \, , \quad \mbox{\rm and} \quad
\Delta = \pi (1,1,10)/12 \, ,
$$
with $N(\Delta)$ equal to $4$, $5$, $6$, and $12$, respectively.

Assume in turn that $r_2 = r_3 = r$ for some $r\in \Q$ with $\frac13
< r <\frac12$. As before, $N(r)\in \{4,5,6,12\}$, which yields
the two triangles
$$
\Delta = \pi (1,2,2)/5 \quad \mbox{\rm and} \quad
\Delta = \pi (2,5,5)/12 \, ,
$$
with $N(\Delta)=5$ and $N(\Delta)=12$, respectively.
In summary, if $\Delta$ is an isosceles rational high school triangle
then $N(\Delta) \in \{3,4,5,6,12\}$. Note that the seven triangles
identified so far are exactly the ones appearing in the left half of
Figure \ref{fig1}.

\medskip

\noindent
{\bf Case II:} $0<r_1<r_2<r_3 < 1${\bf .}

\noindent
Since neither $r_j - r_k$ nor $r_j +r_k$ is an integer, every pair $\bigl( N(r_j),
N(r_k)\bigr)$ must be admissible by Figure \ref{fig3}. Again, this
greatly reduces the number of possible values for $N(r_j)$. First, it is plainly impossible to have $N(r_j)\in
\{15,16,17,20,30\}$.
Second, if $N(r_j)\in \{8,24\}$ for {\em some\/} $j$, then necessarily
$N(r_j)\in \{8,24\}$ for {\em all\/} $j$. But then $r_j N(\Delta)$ is
odd for every $j$ while $N(\Delta)$ is even, and so 
$\sum_{j=1}^3 r_j  = N(\Delta)^{-1} \sum_{j=1}^3 r_j N(\Delta) \ne 1$. This contradiction shows that $N(r_j)\in
\{8,24\}$ is impossible as well. Similarly, if $N(r_j)=5$ for {\em
  some\/} $j$, then $N(r_j)=5$ for {\em all\/} $j$. The latter,
however, is possible only for the two isosceles triangles $\Delta$
with $N(\Delta)=5$ identified in Case I. Overall, it only remains to
consider that case of
\begin{equation}\label{eq4p20}
N(r_j) \in \{2,3,4,6,10,12 \} \, , \quad  j\in \{ 1,2,3 \} \, .
\end{equation}
Again, it is straightforward to see that (\ref{eq4p20}) identifies exactly
seven (non-isosceles) triangles, namely $\Delta = \pi (1,2,3)/6$ with
$N(\Delta)=6$, as well as
\begin{align*}
\Delta & = \pi (1,2,9)/12 \, , \quad  \Delta  =\pi (1,3,8)/12 \, , \quad
\Delta  = \pi (1,4,7)/12 \, , \\
\Delta & = \pi (1,5,6)/12 \, , \quad  \Delta  =\pi (2,3,7)/12 \, , \quad 
\Delta  = \pi (3,4,5)/12 \, ,
\end{align*}
for which $N(\Delta)=12$. Consequently, if $\Delta$ is a non-isosceles rational high
school triangle then $N(\Delta)\in \{6,12\}$. Note that the seven
triangles identified in Case II make up the right half of Figure
\ref{fig1}. Taken together, Case I and II prove the asserted implication
(i)$\Rightarrow$(ii) of Theorem \ref{thm3}. 

To establish the reverse
implication (ii)$\Rightarrow$(i), note first that $N(\Delta)\in
\{3,4,5,6,12\}$ is equivalent to $\Delta$ being equal to exactly one of the
$14$ triangles displayed in Figure \ref{fig1}. Showing that each of
these indeed is a high school triangle requires but a few short,
elementary calculations using the values of $\sin (\pi r)$ provided
(via the factorization of $Q_{N(r)}$) in Figure \ref{fig2}.
\end{proof}

By Theorem \ref{thm3}, a high school triangle $\Delta = (\delta_1,
\delta_2, \delta_3)$ is rational if and only if it equals one of the
$14$ triangles in Figure \ref{fig1}. In any other case, therefore, at least two
of the numbers $r_j = \delta_j/\pi$ are irrational, in fact
transcendental (over $\Q$) by virtue of the Gelfond--Schneider Theorem
\cite[Thm.10.1]{Niven56}. For example, the (Pythagorean) high school triangle with
$\ell_1 :\ell_2 :\ell_3 = 3:4:5$ clearly has $r_3 = \frac12$, while
$r_1 = \pi^{-1} \arccos (4/5)$ and $r_2 = \pi^{-1} \arccos (3/4)$ both
are transcendental \cite[Fact 2]{Calcut10}. Similarly, if $\ell_1 : \ell_2 : \ell_3 = 3:7:8$ then
$r_2 = \frac13$, while $r_1 = \pi^{-1}\arccos (13/14)$ and $r_3 =
\pi^{-1}\arccos (-1/7)$ both are transcendental. On the other hand,
for the high school triangle with $\ell_1 : \ell_2 :\ell_3 = 4:5:6$,
all three numbers $r_j $ are transcendental.

\subsubsection*{Acknowledgements}

The author was supported by an {\sc Nserc} Discovery
Grant. He is much indebted to S.\ Gille and A.\ Weiss for helpful discussions and comments.

\end{document}